\documentclass[12pt]{amsart}

\usepackage{amsfonts, amsthm, amsmath}
\allowdisplaybreaks[4]
\usepackage{pgfplots}

\usepackage{rotating}

\usepackage{tikz}
\usepackage{lmodern}

\usetikzlibrary{decorations.pathmorphing}

\usepackage{xcolor}

\usepackage{graphics}

\usepackage{amssymb}

\usepackage{amscd}

\usepackage[latin2]{inputenc}

\usepackage{t1enc}

\usepackage[mathscr]{eucal}

\usepackage{indentfirst}

\usepackage{enumitem}

\usepackage{enumitem}

\usepackage{graphicx}

\usepackage{graphics}

\usepackage{pict2e}

\usepackage{mathrsfs}

\usepackage{comment}

\usepackage{enumitem}

\usepackage{enumerate}
\usepackage[pagebackref]{hyperref}
\hypersetup{colorlinks=true}
\usepackage{color}
\usepackage{epic}
\usepackage{framed}
\usepackage{mathabx}
\usepackage{booktabs}
\usetikzlibrary{decorations.pathreplacing,calc}
\usepackage{makecell}
\newcolumntype{V}{!{\vrule width 2pt}}

\numberwithin{equation}{section}

\def\blue{\textcolor{blue}}
\def\red{\textcolor{red}}
\def\magenta{\textcolor{magenta}}

\theoremstyle{plain}

\newtheorem{theorem}{Theorem}[section]
\newtheorem{corollary}[theorem]{Corollary}
\newtheorem{proposition}[theorem]{Proposition}

\newtheorem{remark}[theorem]{Remark}
\newtheorem{lemma}[theorem]{Lemma}
\newtheorem{definition}[theorem]{Definition}

\newtheorem{problem}[theorem]{Problem}

\def\S{\mathfrak{S}}

\def\D{\mathcal{D}}
\def\B{\mathcal{B}}
\def\T{\mathcal{T}}
\def\N{\mathbb{N}}

\def\PW{\mathcal{PW}}

\def\OP{\operatorname{op}}
\def\EP{\operatorname{ep}}

\def\ASC{\operatorname{ASC}}
\def\DES{\operatorname{DES}}
\def\LPK{\operatorname{LPK}}
\def\Comp{\operatorname{Comp}}

\def\AR{\mathsf{AR}}
\def\DR{\mathsf{DR}}
\def\odr{\mathsf{odr}}
\def\edr{\mathsf{edr}}
\def\oar{\mathsf{oar}}
\def\ear{\mathsf{ear}}
\def\lpk{\mathsf{lpk}}
\def\mnd{\mathsf{mnd}}
\def\mna{\mathsf{mna}}
\def\mne{\mathsf{mne}}
\def\mnw{\mathsf{mnw}}

\def\M{\mathsf{M}}
\def\Mo{\mathcal{M}}
\def\L{\mathcal{L}}
\def\des{\mathsf{des}}
\def\mk{\mathsf{mark}}
\def\LC{\mathsf{LC}}
\def\RC{\mathsf{RC}}
\def\PT{\mathsf{PT}}
\def\X{\mathcal{X}}
\def\Y{\mathcal{Y}}
\def\sn{\mathrm{sn}}
\newcommand{\der}{{\rm{d}}}

\def\orc{\mathsf{orc}}
\def\olc{\mathsf{olc}}
\def\idr{\mathsf{idr}}
\def\iar{\mathsf{iar}}
\def\a{\alpha}
\def\b{\beta}
\def\larm{\mathsf{larm}}
\def\rarm{\mathsf{rarm}}
\def\st{\mathsf{st}}

\def\H{\mathrm{HOR}}
\def\V{\mathrm{VER}}
\def\Vie{\mathcal{V}}
\def\EXC{\mathrm{EXC}}
\def\WEXC{\mathrm{WEXC}}

\def\Bax{\mathrm{Bax}}
\def\DT{\mathrm{DT}}
\def\DB{\mathrm{DB}}

\def\Comp{\mathrm{Comp}}

\def\EE{\mathcal{EE}}
\def\EO{\mathcal{EO}}
\def\OE{\mathcal{OE}}
\def\OO{\mathcal{OO}}

\def\val{\textbf{val}}
\def\pos{\textbf{pos}}
\def\dia{\diamond}
\def\i{\mathsf{i}}
\def\t{\mathsf{t}}

\def\pk{\mathsf{pk}}

\begin{document}
\title[Parity statistics on restricted permutations]{Parity statistics on restricted permutations and the Catalan--Schett polynomials}

\author[Z. Lin]{Zhicong Lin}
\address[Zhicong Lin]{Research Center for Mathematics and Interdisciplinary Sciences,  Shandong University \& Frontiers Science Center for Nonlinear Expectations, Ministry of Education, Qingdao 266237, P.R. China}
\email{linz@sdu.edu.cn}

\author[J. Liu]{Jing Liu}
\address[Jing Liu]{Research Center for Mathematics and Interdisciplinary Sciences,  Shandong University \& Frontiers Science Center for Nonlinear Expectations, Ministry of Education, Qingdao 266237, P.R. China}
\email{lsweet@mail.sdu.edu.cn}

\author[S.H.F. Yan]{Sherry H.F. Yan}
\address[Sherry H.F.  Yan]{Department of Mathematics,
Zhejiang Normal University, Jinhua 321004, P.R. China}
\email{hfy@zjnu.cn}

\date{\today}

\begin{abstract}
Motivated by Kitaev and Zhang's recent work on non-overlapping ascents in stack-sortable permutations and Dumont's permutation interpretation of the Jacobi elliptic functions, we investigate some parity statistics on restricted permutations. Some new related bijections are constructed and two  refinements of  the generating function for descents over $321$-avoiding permutations due to Barnabei, Bonetti and Silimbanian are obtained. In particular, an open problem of Kitaev and Zhang about non-overlapping ascents on $321$-avoiding permutations is solved and several combinatorial interpretations for the Catalan--Schett polynomials are found. The stack-sortable permutations are at the heart of our approaches. 
\end{abstract}


\maketitle

\section{Introduction}
Let $\S_n$ be the set of all permutations of $[n]:=\{1,2,\ldots,n\}$.  A permutation $\pi=\pi_1\pi_2\cdots\pi_n\in\S_n$ is said to {\em avoid} pattern $\sigma\in\S_k$ if there does not exist $ i_1<i_2<\dots<i_k$ such that the  subsequence $\pi_{i_1}\pi_{i_2}\dots\pi_{i_k}$ of $\pi$ is order isomorphic to $\sigma$. Let
$$
\S_n(\sigma):=\{\pi\in\S_n: \pi\text{ avoids } \sigma\}.
$$
An element in $\S_n(\sigma)$ is called a {\em$\sigma$-avoiding permutation}. One of the classical  enumerative results in pattern avoiding permutations, attributed to MacMahon and Knuth (see~\cite{Kit2}), is that $|\S_n(\sigma)|=C_n$ for
each pattern $\sigma\in\S_3$, where $C_n=\frac{1}{n+1}{2n \choose n}$ is the $n$-th {\em Catalan number}. Because of Knuth's work~\cite{Knu},  the class of $231$-avoiding permutations was also known as {\em stack-sortable permutations}. 

The MacMahon--Knuth result has been refined several times in the literature. Robertson, Saracino and Zeilberger~\cite{RSZ} refined this result by proving that the distribution of ``number of fixed points'' is the same in 321-avoiding as in 132-avoiding permutations. Elizalde and Pak~\cite{EP} further refined Robertson et al.'s result by  taking into account the ``number of excedances'' in a permutation. Barnabei, Bonetti and Silimbanian~\cite{bbs} showed that 
\begin{equation}\label{eq:BBS}
(t^2x-t^2+t)A^2+(2t^2x^2-2t^x+2xt-x)A+t^2x^3-t^2x^2+tx^2=0,
\end{equation}
where $A$ is the enumerator of $321$-avoiding permutations by the length (variable $t$) and the number of descents (variable $x$).  The main objective of  this paper is to present new refinements, from the  aspects of both bijective combinatorics and the generating functions,  of the MacMahon--Knuth result using two classes of parity statistics on permutations. 

We begin by introducing the first class of parity statistics. 
Given  $\pi\in\S_n$, an index $i\in[n-1]$ is  called an {\em ascent} of $\pi$ if $\pi_i<\pi_{i+1}$ and a {\em descent} of $\pi$ if $\pi_i>\pi_{i+1}$. Let $\ASC(\pi)$ (resp., $\DES(\pi)$) be the set of ascents (resp., descents) of $\pi$. 
To a subset $S\subseteq [n-1]$ with elements $s_1<s_2<\cdots <s_k$, we associate the multiset  
$$\M(S) =
\begin{cases} 
\,\{s_1, s_2-s_1,\dots, s_k-s_{k-1}, n-s_k\}, \quad&\text{if $S\neq\emptyset$;}\\
\,\{n\},&\text{if $S=\emptyset$}.
\end{cases}
$$ Then, $\DR(\pi):=\M(\ASC(\pi))$ is  the {\em descending run multiset} of $\pi$ and $\AR(\pi):=\M(\DES(\pi))$ is the {\em ascending  run multiset} of $\pi$. Note that $\DR(\pi)$ records the lengths of the decreasing runs of $\pi$, while $\AR(\pi)$ records the lengths of the increasing runs of $\pi$. For instance, if $\pi=318972456\in\S_9$, then we have $\DR(\pi)=\{3,2,1^4\}$ as $\ASC(\pi)=\{2,3,6,7,8\}$, while  $\AR(\pi)=\{4,3,1^2\}$ as $\DES(\pi)=\{1,4,5\}$. Introduce four parity statistics of $\pi$ as 
\begin{itemize}
\item$\odr(\pi)$, the number of odd elements in  $\DR(\pi)$;
\item $\edr(\pi)$, the number of even elements in $\DR(\pi)$;
\item$\oar(\pi)$, the number of odd elements in $\AR(\pi)$;
\item $\ear(\pi)$, the number of even elements in  $\AR(\pi)$.
\end{itemize}
Continuing with the running example, we have $\odr(\pi)=5$, $\edr(\pi)=1$, $\oar(\pi)=3$ and $\ear(\pi)=1$.

Another class of parity statistics that we consider are the even/odd left peaks of permutations. Recall that a {\em left peak} of a permutation $\pi\in\S_n$ is a value $\pi_i$, $i\in[n-1]$, such that $\pi_{i-1}<\pi_i>\pi_{i+1}$ with the convention $\pi_0=0$. Let $\LPK(\pi)$ be the set of left peaks of $\pi$. Introduce the two parity statistics of $\pi$ as
\begin{itemize}
\item $\lpk_o(\pi)$, the number of odd elements in $\LPK(\pi)$;
\item $\lpk_e(\pi)$, the number of even elements in $\LPK(\pi)$.
\end{itemize}
For example, if $\pi=3271654$, then $\LPK(\pi)=\{3,6,7\}$,  so $\lpk_o(\pi)=2$ and $\lpk_e(\pi)=1$. 

Our motivation to consider these parity statistics on restricted permutations comes from the recent work by Kitaev and Zhang~\cite{KZ}. The notion of the {\em maximum number of non-overlapping occurrences of a consecutive pattern} in a permutation was  first considered by Kitaev~\cite{Kit}. Kitaev and Zhang~\cite{KZ} focused on the {\em maximum number of non-overlapping descents} (resp., {\em ascents}), denoted $\mnd(\pi)$ (resp., $\mna(\pi)$), of a permutation $\pi$. In other words, we have 
$$
\mnd(\pi)=\sum_{i\in\DR(\pi)}\big\lfloor\frac{i}{2}\big\rfloor\quad\text{and}\quad \mna(\pi)=\sum_{i\in\AR(\pi)}\big\lfloor\frac{i}{2}\big\rfloor. 
$$ 
Notice that 
\begin{equation}\label{ov:mna}
\mnd(\pi)=\frac{n-\odr(\pi)}{2}\quad\text{and}\quad\mna(\pi)=\frac{n-\oar(\pi)}{2}.
\end{equation}
The main results in~\cite{KZ} are outlined as follows.
\begin{theorem}[Kitaev and Zhang~\cite{KZ}]\label{thm:kz}
 The following two results hold:
\begin{itemize}
\item[(i)] the pair $(\mna,\mnd)$ is symmetric over $\S_n(231)$;
\item[(ii)] $|\{\pi\in\S_n(231): \mnd(\pi)=k\}|=\frac{1}{n+1}{n+1\choose 2k+1}{n+k\choose k}$.
\end{itemize}
\end{theorem}
Let $\des(\pi)=|\DES(\pi)|$ be the number of descents of $\pi$. As $\mnd(\pi)=\des(\pi)$ for any $\pi\in\S_n(321)$, the distribution of ``$\mnd$'' over $321$-avoiding permutations has been computed by Barnabei et al.~\cite{bbs}. Kitaev and Zhang~\cite{KZ} posed the following open problem, resolving which would complete the enumeration of ``$\mnd$'' and ``$\mna$'' over permutations avoiding a pattern of length $3$. 

\begin{problem}[Kitaev and Zhang~\cite{KZ}]\label{PB:kz}
Find the distribution of ``$\mna$'' over $321$-avoiding permutations. 
\end{problem}

The refinement of Catalan numbers 
\begin{equation}\label{ref:cat}
\frac{1}{n+1}{n+1\choose 2k+1}{n+k\choose k}
\end{equation} appears in Theorem~\ref{thm:kz} also has other known combinatorial interpretations. 
A {\em binary tree} is a special type of rooted tree in which every internal node  has either one left child  or one right child  or both. Let $\B_n$ be the set of all binary trees with $n$ nodes.  Given a binary tree, its {\em left chain} (resp.,~{\em right chain}) is any maximal path composed of only left (resp.,~right) edges. 
The {\em order} of a left/right chain is the number of its nodes. Let $\LC(T)$ (resp.,~$\RC(T)$) be the multiset of all orders of the left (resp.,~right) chains of a tree $T\in\B_n$. 
For instance, the binary tree in Fig.~\ref{vie} has $5$ right chains, which are $4-3-2$, $9-8$, $7-6$,  $1$ and $5$, and so $\RC(T)=\{3,2^2,1^2\}$. Define two statistics of $T$  as 
$$
\X(T)=\sum_{i\in\LC(\pi)}\big\lfloor\frac{i}{2}\big\rfloor\quad\text{and}\quad \Y(T)=\sum_{i\in\RC(\pi)}\big\lfloor\frac{i}{2}\big\rfloor. 
$$ 
For the running example tree $T$, we have $\X(T)=2$ and $\Y(T)=3$. Sun~\cite{Sun} proved combinatorially that 
the number~\eqref{ref:cat} enumerates binary trees $T\in\B_n$ with $\X(T)=k$. Kitaev and Zhang~\cite{KZ} constructed a bijection between $\{T\in\B_n:\X(T)=k\}$ and $\{\pi\in\S_n(231): \mnd(\pi)=k\}$, providing the first proof of Theorem~\ref{thm:kz}~(ii).  

The original impetus of this work lies in an unexpected connection between the joint distribution of $(\X,\Y)$ over binary trees and the multiset Schett polynomials introduced by Ma and the first author~\cite{LM}.  The {\em Jacobi elliptic function} (see~\cite{LM}) $\sn(u,\alpha)$ may be defined by the inverse of an elliptic integral:
$$
\sn(u,\alpha)=y\quad\text{iff}\quad u=\int_0^y\frac{\der t}{\sqrt{(1-t^2)(1-\alpha^2t^2)}},
$$
where $\alpha\in(0,1)$ is a real number. When $a=1$, $\sn(u,\alpha)$ becomes the sine function $\sin(u)$. For a multiset $M$, a binary tree whose nodes are labeled exactly by $M$ such that each child node receives a label weakly greater than its parent is called a  {\em weakly increasing binary trees} on $M$. According to the work in~\cite{LM,LLWZ}, the {\em multiset Schett polynomials} $S_M(x,y)$, which extends the Jacobi elliptic function from sets to multisets,  can be interpreted as 
$$
S_M(x,y)=\sum_{T\in\B_M} x^{\olc(T)}y^{\orc(T)},
$$
where  $\B_M$ is the set of weakly increasing binary trees on $M$ and $\olc(T)$ (resp., $\orc(T)$) denotes the number of left (resp.,~right) chains of $T$ with odd orders. 
Note that weakly increasing binary trees on $[n]$ are exactly {\em increasing binary trees} on $[n]$, while weakly increasing trees on $\{1^n\}$ are in obvious bijection with  $\B_n$.  For convenience, we write $S_M(x,y)$ as $S_n(x,y)$ (resp.,~$C_n(x,y)$) when $M=[n]$ (resp.,~$M=\{1^n\}$). Then, $S_n(x,y)$ are the classical Schett polynomials (see~\cite{Schet,Dum,LM}) that specialize to the Jacobi elliptic function $\sn(u,\alpha)$. The first combinatorial interpretation of $S_n(x,y)$, which is in terms of $(\lpk_o,\lpk_e)$ on permutations,  was found by Dumont~\cite{Dum}.  The bivariate extension of Catalan numbers $C_n(x,y)$ will be named the {\em Catalan--Schett polynomials}. The first few values of $C_n(x,y)$ are:
\begin{align*}
C_1(x,y)&=xy,\quad C_2(x,y)=x^2+y^2, \quad C_3(x,y)=x^{3}y+xy^{3}+3xy,\\
C_4(x,y)&={x}^{4}+8{y}^{2}{x}^{2}+{y}^{4}+2{x}^{2}+2{y}^{2},\\
C_5(x,y)&={x}^{5}y+5{x}^{3}{y}^{3}+x{y}^{5}+15{x}^{3}y+15x{y}^{3}+5xy\\
C_6(x,y)&={x}^{6}+27{x}^{4}{y}^{2}+27{y}^{4}{x}^{2}+{y}^{6}+8{x}^{4}+54{
y}^{2}{x}^{2}+8{y}^{4}+3{x}^{2}+3{y}^{2}.
\end{align*} 
As
$$
\X(T)=\frac{n-\olc(T)}{2}\quad\text{and}\quad\Y(T)=\frac{n-\orc(T)}{2}
$$
for any $T\in\B_n$, we have 
$$
C_n(x,y)=\sum_{T\in\B_n} x^{n-2\X(T)}y^{n-2\Y(T)}. 
$$
Our first main result provides a new interpretation of $C_n(x,y)$ in terms of stack-sortable permutations.

\begin{theorem}\label{thm:sch}
There is a bijection $\Upsilon: \pi\mapsto T$ between $\S_n(231)$ and $\B_n$ such that 
\begin{equation}\label{dr:ar}
(\DR(\pi),\AR(\pi^{-1}))=(\LC(T),\RC(T)).
\end{equation}
Consequently, 
\begin{equation}\label{X:Y}
\sum_{\pi\in\S_n(231)}x^{\mnd(\pi)}y^{\mna(\pi^{-1})}=\sum_{T\in\B_n} x^{\X(T)}y^{\Y(T)}
\end{equation}
and there follows the permutation interpretation for the Catalan--Schett polynomials 
$$
C_n(x,y)=\sum_{\pi\in\S_n(231)}x^{\odr(\pi)}y^{\oar(\pi^{-1})}.
$$
\end{theorem}
Combining $\Upsilon$ with several known bijections, we can prove another interpretation of the Catalan--Schett polynomials in terms of $321$-avoiding permutations, with the role of descents replaced by excedances; see Theorem~\ref{thm:Phi2}.

There is another interpretation of the number~\eqref{ref:cat} in terms of plane trees found by Callan~\cite{Cal}. Recall that a {\em plane
tree} is a rooted tree in which the children of each node are linearly ordered.
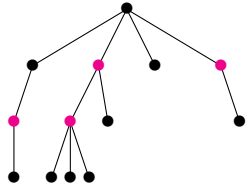
\begin{figure}
\centering
\begin{tikzpicture}[scale=0.25]
\node at (20,20) {$\bullet$};
\node at (15,17) {$\bullet$};

\node at (21.5,17){$\bullet$};
\node at  (26,14) {$\bullet$};

\node at (16,11) {$\bullet$};
\node at (17,11) {$\bullet$};
\node at (18,11) {$\bullet$};

\node at (19,14) {$\bullet$};

\node at (14,11) {$\bullet$};

\draw[-] (25,17) to (26,14);
\draw[-] (20,20) to (15,17);
\draw[-] (20,20) to (18.5,17);
\draw[-] (20,20) to (21.5,17);
\draw[-] (20,20) to (25,17);
\draw[-] (15,17) to (14,14);
\draw[-] (14,14) to (14,11);
\draw[-] (18.5,17) to (17,14);
\draw[-] (18.5,17) to (19,14);
\draw[-] (17,14) to (16,11);
\draw[-] (17,14) to (17,11);
\draw[-] (17,14) to (18,11);

\node at (17,14) {\magenta{$\bullet$}};
\node at (18.5,17) {\magenta{$\bullet$}};
\node at (25,17) {\magenta{$\bullet$}};
\node at (14,14){\magenta{$\bullet$}};
\end{tikzpicture}
\caption{A plane tree with four marked nodes (in magenta).\label{p:tree}}
\end{figure}
An internal node (other than the root) in a plane tree is {\em marked} if it has a leaf as its child. See Fig.~\ref{p:tree} for a plane tree with four marked nodes.  Let $\T_n$ be the set of plane trees with $n$ edges and let $\mk(T)$ be the number of marked nodes in a tree $T\in\T_n$. Callan~\cite{Cal} proved that the number~\eqref{ref:cat} counts plane trees $T\in\T_n$ with $\mk(T)=k$. In order to provide the second proof of Theorem~\ref{thm:kz}~(ii), Zhang and Kitaev~\cite{KZ} constructed a recursive bijection, involving  five different cases, between  $\{T\in\T_n:\mk(T)=k\}$ and $\{\pi\in\S_n(231): \mnd(\pi)=k\}$. Here we provide a case-free recursive bijection which requires extra decompositions of plane trees. 

\begin{theorem}\label{bij:pt-231}
There is a bijection $\vartheta: \T_n \to \S_n(231)$ such that $\mk(T)=\mnd(\vartheta(T))$ for  any $T \in \T_n$.
\end{theorem}

Various bijections and equidistributions between $\S_n(231)$ and $\S_n(321)$ have already been investigated  in the literature; see  Kitaev's monograph~\cite[Chap.~4]{Kit2} for a survey. Inspired by Dumont's permutation interpretation for the Schett polynomials, we find the following unexpected equidistribution. 

\begin{theorem}\label{bij:LPK}
There is a bijection $\Psi: \S_n(321)\rightarrow \S_n(231)$ preserving the set-valued statistic ``$\LPK$''. Consequently, 
\begin{equation}\label{eo:lpk}
\sum_{\pi\in\S_n(321)}x^{\lpk_e(\pi)}y^{\lpk_o(\pi)}=\sum_{\pi\in\S_n(231)}x^{\lpk_e(\pi)}y^{\lpk_o(\pi)}. 
\end{equation}
\end{theorem}

Note that $\LPK(\pi)=\DES(\pi)$ for any $\pi\in\S_n(321)$. Using the bijection $\Psi$ and a bijection of Krattenthaler~\cite{Kra} between $321$-avoiding permutations and Dyck paths, we prove the following two refinements of Barnabei et al.'s generating function formula~\eqref{eq:BBS}.

\begin{theorem}\label{thm:gf}
Let 
$$G(t,x,y)=\sum_{n\geq1}t^n\sum_{\pi\in\S_n(321)}x^{\oar(\pi)}y^{\ear(\pi)}\quad\text{and}\quad M(t,x,y)=\sum_{n\geq1}t^n\sum_{\pi\in\S_n(321)}x^{\lpk_e(\pi)}y^{\lpk_o(\pi)}.
$$
Then both generating functions $G$ and $M$ are algebraic. 
\begin{itemize}
\item[(i)] $G$ satisfies an algebraic equation of degree $4$; see~\eqref{alg:gf1}. In particular, if we set $A=G(t,x,1)$, then $A$ satisfies 
\begin{align*}
\qquad(t^5+4t^4x+4t^3)A^4 = & (8t^3x^2+8t^2x-4t^5-14t^4x-16t^3)A^3-(6t^5+2t^4x^3\\
&+18t^4x+22t^3+5t^2x^3+5tx^2-13t^3x^2-17t^2x-t)A^2\\
&-(4t^5+4t^4x^3+10t^4x+ 12t^3+3tx^2+x-2t^3x^4 - 2t^3x^2\\
&-8t^2x-tx^4 -2t-x^3)A-( t^5+2t^4x^3 +2t^4x+t^3x^6  \\
&+3t^3x^2+2t^3+2t^2x^5+ t^2x+tx^4  -2t^3x^4 -3t^2x^3 -tx^2),
\end{align*}
which solves Problem~\ref{PB:kz} in view of~\eqref{ov:mna}. 
\item[(ii)] $M$ satisfies an algebraic equation of degree $6$; see~\eqref{alg:gf2}.
\end{itemize}
\end{theorem}

The rest of this paper is organized as follows. Section~\ref{sec:bj} is devoted to the constructions of the bijections  $\Upsilon$, $\vartheta$ and $\Psi$, as well as some other related bijections. Utilizing two bijections presented in  Section~\ref{sec:bj}, we prove Theroem~\ref{thm:gf} in Section~\ref{sec:gf}.

\section{Bijections}\label{sec:bj}
In this section, we present several bijections in the garden of Catalan numbers, some of which will also be useful in computing generating functions for two enumerators of $321$-avoiding permutations in next section. Theorems~\ref{thm:sch},~\ref{bij:pt-231} and~\ref{bij:LPK} will be proved, as well as an interpretation of the Catalan--Schett polynomials in terms of $321$-avoiding permutations.

\subsection{The bijection $\Upsilon$ from stack-sortable permutations to binary trees}
The construction of  $\Upsilon$ is based on a  decomposition of stack-sortable permutations which was first used in~\cite{Lin2}. 

For $\sigma\in\S_k$ and $\pi\in\S_l$, 
 define the {\em direct sum} of $\sigma$ and $\pi$, denoted $\sigma\oplus\pi$, to be a permutation in $\S_{k+l}$ whose $i$-th letter is
$$
(\sigma\oplus \pi)_i=
\begin{cases}
\sigma_i, &\text{for $1\leq i\leq k$};\\
\pi_{i-k}+k, &\text{for $k+1\leq i\leq k+l$}.
\end{cases}
$$
For example, we have $312\oplus2143=312\magenta{5476}$. The following decomposition of stack-sortable permutations will be used frequently throughout this paper. 
\begin{lemma}[Lin-Wang-Zhao~\cite{Lin2}]\label{dec:231}
For any $\pi\in\S_n(231)$ with $\pi_1=k$, $\pi$ can be decomposed as $\pi=(k\cdot\pi')\oplus\pi''$, where $\pi'\in\S_{k-1}(231)$ and $\pi''\in\S_{n-k}(231)$. 
\end{lemma}

We call the decomposition of $\pi$ as $(k\cdot\pi')\oplus\pi''$ in Lemma~\ref{dec:231} the {\em first letter decomposition} of $\pi$.
For example, $\pi=3125476\in\S_7(231)$ is decomposed as $\pi=(3\cdot12)\oplus2143$.

Next, we need a transformation $T\mapsto T^*$ on a class of special binary trees, namely these trees whose root has only left branch. Given such a binary tree $T$ with root $r$ and the (only)  left child of $r$ is $r'$, then let $T^*$ be the binary tree obtained from $T$ by first cutting the right branch of $r'$ and then attaching it as a right branch at $r$. Note that the node $r'$ in $T^*$ has no right branch. 

\begin{figure}
\centering
\begin{tikzpicture}[scale=0.24]

\node at (20,20){$\bullet$};\draw[-] (20,20) to (11,11);
\node at (20,20.9){$r$};\node at (17,18){$r'$};
\node at (17,17){$\bullet$};\node at (14,14){$\bullet$};\node at (11,11){$\bullet$};
\draw[-] (11,11) to (13,8);\draw[-] (14,14) to (17,11);\draw[-] (14,8) to (17,11);
\node at (13,8){$\bullet$};\node at (17,11){$\bullet$};\node at (14,8){$\bullet$};

\magenta{\draw[-] (17,17) to (23,11);
\node at (20,14){$\bullet$};
\node at (23,11){$\bullet$};
\draw[-] (17,5) to (23,11);
\node at (20,8){$\bullet$};
\node at (17,5){$\bullet$};
\draw[-] (20,8) to (23,5);
\node at (23,5){$\bullet$};
\draw[-] (23,5) to (20,2);
\node at (20,2){$\bullet$};
}
\node at (17,17){$\bullet$};

\node at (27,13){$\mapsto$};
\node at (40,20){$\bullet$};\draw[-] (40,20) to (31,11);
\node at (40,20.9){$r$};\node at (37,18){$r'$};
\node at (37,17){$\bullet$};\node at (34,14){$\bullet$};\node at (31,11){$\bullet$};
\draw[-] (31,11) to (33,8);\draw[-] (34,14) to (37,11);\draw[-] (34,8) to (37,11);
\node at (33,8){$\bullet$};\node at (37,11){$\bullet$};\node at (34,8){$\bullet$};

\magenta{\draw[-] (40,20) to (46,14);
\node at (43,17){$\bullet$};
\node at (46,14){$\bullet$};
\draw[-] (40,8) to (46,14);
\node at (43,11){$\bullet$};
\node at (40,8){$\bullet$};
\draw[-] (43,11) to (46,8);
\node at (46,8){$\bullet$};
\draw[-] (46,8) to (43,5);
\node at (43,5){$\bullet$};
}
\node at (40,20){$\bullet$};
\end{tikzpicture}
\caption{An example of the transformation  $T\mapsto T^*$. \label{transfor}}
\end{figure}

{\bf The construction of $\Upsilon$.} The mapping $\Upsilon: \S_n(231)\rightarrow \B_n$ will be constructed recursively according to two cases. Given $\pi\in\S_n(231)$, if $\pi=1$, then $\Upsilon(\pi)$ is the unique binary tree with only one node. Otherwise, by Lemma~\ref{dec:231}, $\pi=(k\cdot\pi')\oplus\pi''$, where $k=\pi_1$, $\pi'\in\S_{k-1}(231)$ and $\pi''\in\S_{n-k}(231)$. We distinguish two cases according to $\pi'$ is empty or not. 
\begin{itemize}
\item[(1)] \underline{$\pi'$ is empty, i.e., $k=1$}. Then define $\Upsilon(\pi)$ to be the tree obtained by attaching $\Upsilon(\pi'')$ as the right branch of a new node $r$. Note that $r$ is the root of $\Upsilon(\pi)$ and $r$ has no left branch. 
\item[(2)] \underline{$\pi'$ is nonempty}. We perform the following two steps: 
\begin{itemize}
\item[i)] Construct the tree $T$ by attaching $\Upsilon(\pi')$ as the left branch of a new node $r$. Suppose that $r'$ is the left child of the root $r$ in $T$. 
\item[ii)] Note that in $T^*$, the node $r'$ (left child of the root $r$) has no right branch.  Now, $\Upsilon(\pi)$ is obtained from $T^*$ by attaching $\Upsilon(\pi'')$ at $r'$ as right branch. In this case, the root $r$ of $\Upsilon(\pi)$ has left child. 
\end{itemize}
\end{itemize}
See Fig.~\ref{ex:upsilon} for an example of $\Upsilon$.
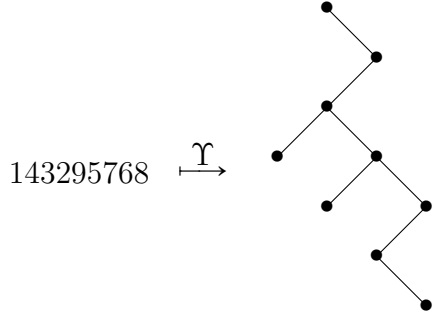
\begin{figure}
\centering
\begin{tikzpicture}[scale=0.22]
\node at (0,10){$143295768$};\draw[-] (15,20) to (18,17);\draw[-] (12,11) to (18,17);
\node at (7.5,10){$\longmapsto$};\node at (15,20){$\bullet$};\node at (18,17){$\bullet$};
\node at (7.5,11.2){$\Upsilon$};
\node at (15,14){$\bullet$};\node at (12,11){$\bullet$};
\draw[-] (15,14) to (21,8);\node at (18,11){$\bullet$};\draw[-] (18,11) to (15,8);
\node at (15,8){$\bullet$};
\node at (21,8){$\bullet$};\draw[-] (21,8) to (18,5);\draw[-] (21,2) to (18,5);
\node at (18,5){$\bullet$};\node at (21,2){$\bullet$};
\end{tikzpicture}
\caption{An example of $\Upsilon$. \label{ex:upsilon}}
\end{figure}
\begin{lemma}\label{bij:upsi}
The mapping  $\Upsilon$ is a bijection between $\S_n(231)$ and $\B_n$. 
\end{lemma}
\begin{proof}
Given $T\in\B_n$, if the root of $T$ has no left child, then $T$ is the image under the first case of $\Upsilon$. Otherwise, the root of $T$ has  left child and $T$ is the image under the second case of $\Upsilon$.
As the transformation  $T\mapsto T^*$ is reversible, we see that the mapping  $\Upsilon$ is a bijection between $\S_n(231)$ and $\B_n$ by induction on $n$. 
\end{proof}
 For $\pi\in\S_n$, define 
\begin{itemize}
\item $\idr(\pi)$, the length of the {\em initial descending run} of $\pi$;
\item $\iar(\pi)$,  the length of the {\em initial ascending run} of $\pi$.
\end{itemize}
For $T\in\B_n$, the left (resp., right) chain starting from the root
of $T$ is called the {\em left} (resp.,~{\em right}) {\em arm} of $T$ and we denote by $\larm(T)$ (resp., $\rarm(T)$) the   order of the left (resp.,~right) arm of $T$.
The bijection $\Upsilon$ has the following important  feature. 
\begin{lemma}
Let $\pi\in\S_n(231)$ and $T=\Upsilon(\pi)$. Then, 
\begin{equation}\label{idr:iar}
\idr(\pi)=\larm(T)\quad\text{and}\quad\iar(\pi^{-1})=\rarm(T).
\end{equation}
\end{lemma}
\begin{proof}
We prove this lemma by induction on $n$. Suppose that the first letter decomposition of $\pi$ is $(k\cdot\pi')\oplus\pi''$. Let $T'=\Upsilon(\pi')$ and $T''=\Upsilon(\pi'')$. If $\pi'$ is empty, then $T$ is the tree obtained by attaching $\Upsilon(\pi'')$ as the right branch of a new node, thus 
$$
\idr(\pi)=1=\larm(T)\quad\text{and}\quad\iar(\pi^{-1})=\iar((\pi'')^{-1})+1=\rarm(T'')+1=\rarm(T),
$$
where the penultimate equality follows by induction. If $\pi'$ is not empty, then 
\begin{equation}\label{eq:case2}
\idr(\pi)=\idr(\pi')+1\quad\text{and}\quad\iar(\pi^{-1})=\iar((\pi')^{-1}).
\end{equation}
By the second case of the construction of $T$ from $T'$ and $T''$, we have 
$$
\larm(T)=\larm(T')+1\quad\text{and}\quad\rarm(T)=\rarm(T'),
$$
which comparing with~\eqref{eq:case2} leads to~\eqref{idr:iar} by induction.  This completes the proof of the lemma. 
\end{proof}
We are now ready to prove Theorem~\ref{thm:sch}.
\begin{proof}[{\bf Proof of Theorem~\ref{thm:sch}}]
In view of Lemma~\ref{bij:upsi}, it remains to show that $\Upsilon$ satisfies~\eqref{dr:ar}, namely
$$
(\DR(\pi),\AR(\pi^{-1}))=(\LC(T),\RC(T)).
$$ 
We proceed to prove this by induction on $n$. 

Suppose that the first letter decomposition of $\pi$ is $(k\cdot\pi')\oplus\pi''$, $T'=\Upsilon(\pi')$ and $T''=\Upsilon(\pi'')$. If $\pi'$ is empty, then
$$
\DR(\pi)=\{1\}+\DR(\pi'')
$$
and
$$\AR(\pi^{-1})=\AR((\pi'')^{-1})-\{\iar((\pi'')^{-1})\}+\{\iar((\pi'')^{-1})+1\}.
$$
Here ``$+$'' and ``$-$'' are multiset plus and minus. In this case, $T$ is the tree obtained by attaching $\Upsilon(\pi'')$ as the right branch of a new node, thus
$$
\LC(T)=\{1\}+\LC(T'')
$$
and 
$$
\RC(T)=\RC(T'')-\{\rarm(T'')\}+\{\rarm(T'')+1\}.
$$
The desired property~\eqref{dr:ar} in this case then follows by induction and~\eqref{idr:iar}. 

If $\pi'$ is not empty, then
$$
\DR(\pi)=\DR(\pi')-\{\idr(\pi')\}+\{\idr(\pi')+1\}+\DR(\pi'')
$$
and
$$\AR(\pi^{-1})=\AR((\pi')^{-1})+\AR((\pi'')^{-1})-\{\iar((\pi'')^{-1})\}+\{\iar((\pi'')^{-1})+1\}.
$$
On the other hand, $T$ is constructed from $T'$ and $T''$ in the second case of $\Upsilon$, from which we see
$$
\LC(T)=\LC(T')-\{\larm(T')\}+\{\larm(T')+1\}+\LC(T'')
$$
and 
$$
\RC(T)=\RC(T')+\RC(T'')-\{\rarm(T'')\}+\{\rarm(T'')+1\}.
$$
It then follows by induction and~\eqref{idr:iar} that property~\eqref{dr:ar} holds when $\pi'$ is not empty. This completes the proof of the theorem. 
\end{proof}

\subsection{A bijection from stack-sortable permutations to pairs of non-intersecting walks}

 A {\em walk} is a lattice path in the grid $\N^2$  starts at the origin and consists of {\em north} step $N=(1,0)$ and {\em east} step $E=(0,1)$. The {\em size} of a walk is one plus the number of steps it contains, i.e., the number of lattice points it contains. A pair of walks $(\mu,\nu)$ with the same endpoint is said to be {\em non-intersecting} if $\mu$ is weakly above $\nu$. Denote by $\PW_n$ the set of non-intersecting walk pairs of size $n$. In this subsection, we construct a natural bijection from stack-sortable permutations to pairs of non-intersecting walks that will be useful in next subsection. 
 
For a walk $\mu$, let $\H(\mu)$ (resp., $\V(\mu)$)  denote the set of all positions of the east (resp., north) steps of $\mu$. 
 For example, if $\mu$ is the top walk in Fig.~\ref{vie}, then $\H(\mu)=\{2,3,6,8\}$ and $\V(\mu)=\{1,4,5,7\}$.
 
 \begin{theorem}\label{thm:theta}
 There is a bijection $\Theta: \pi\mapsto(\mu,\nu)$ between $\S_n(231)$ and $\PW_n$ such that 
 \begin{equation}\label{da:hv}
 (\DES(\pi),\ASC(\pi^{-1}))=(\H(\nu),\V(\mu)).
 \end{equation}
 \end{theorem}
 
 The bijection $\Theta$ is the composition of the classical bijection $\varphi$ between $\S_n(231)$ and $\B_n$ (see~\cite[Sec.~1.5]{EC1}) and the bijection $\Vie$ from $\B_n$ to $\PW_n$ due to Pr\'{e}ville-Ratalle and Viennot~\cite{Vie}.  
 
The bijection $\varphi$ is defined recursively.  It is well known that each $\pi\in\S_n(231)$ with $\pi_{k+1}=n$ for some $0\leq k<n$ can be decomposed into a pair $(\a,\b)$, where 
 $$
 \a=\pi_1\pi_2\cdots\pi_{k}\in\S_{k}(231) 
 $$
and
$$
\b=(\pi_{k+2}-k)(\pi_{k+3}-k)\cdots(\pi_n-k)\in\S_{n-k-1}(231).
$$
The triple $(\a,n,\b)$ will be called the {\em greatest letter decomposition} of $\pi$. 
With this decomposition of $\pi$,  $\varphi(\pi)$ is defined to be the binary tree whose root has $\varphi(\a)$ as left branch and  $\varphi(\b)$ as right branch. See Fig.~\ref{vie} for an example of $\varphi$.

\begin{figure}
\centering
\begin{tikzpicture}
[inner sep=.6mm, level distance=1cm,
  level 1/.style={sibling distance=4cm},
  level 2/.style={sibling distance=2cm},
  level 3/.style={sibling distance=1.25cm},
  level 4/.style={sibling distance=0.5cm},
 level 4/.style={sibling distance=0.25cm},
]
 \node[circle,draw,label=above:{\large 9}](9){}
    child{node[circle,draw,label=above:{\large 4}](4){}
      child{node[circle,draw,label=below:{\large 1}](1){}}
      child{node[circle,draw,label=above:{\large 3}](3){}
        child[missing] {node {}}
        child{node[circle,draw,label=below:{\large 2}](2){}}
      }
    }
    child{node[circle,draw,label=above:{\large 8}](8){}
      child{node[circle,draw,label=above:{\large 7}](7){}
        child{node[circle,draw,label=below:{\large 5}](5){}
          child[missing]{node{}}
          child[missing] {node {}}
        }
        child{node[circle,draw,label=below:{\large 6}](6){}
          child[missing] {node {}}
          child[missing] {node {}}
        }
      }
      child[missing] {node {}}
    };

\draw (4)--(3)  node[red, midway,below left,font=\small] {$E$} node[red, midway,above right,font=\small] {$\bar{E}$};
\draw (3)--(2)  node[red, midway,below left,font=\small] {$E$} node[red, midway, above right,font=\small] {$\bar{E}$};
\draw  (9)--(8) node[red, midway,below left,font=\small] {$E$} node[red,midway,above right,font=\small] {$\bar{E}$};
\draw(7)--(6)  node[red,pos=0.8, below left] {$E$} node[red, midway,above right,font=\small] {$\bar{E}$};

\draw (9) -- (4) node[blue, midway,above left,font=\small] {$N$} node[blue, midway,below right,font=\small] {$\bar{N}$};
\draw (4)--(1) node[blue, midway,above left,font=\small] {$N$} node[ blue,midway, right,font=\small] {$\bar{N}$};
\draw (8)--(7) node[blue, midway,above left,font=\small] {$N$} node[blue, midway,below right,font=\small] {$\bar{N}$};
\draw (7)--(5) node[blue, midway,above left,font=\small] {$N$} node[blue, midway, right,font=\small ] {$\bar{N}$};

\node[below=4cm,text=orange] at (9) {$\mu=\bar{N} \bar{E} \bar{E} \bar{N} \bar{N} \bar{E} \bar{N} \bar{E}$};
\node[below=4.5cm,text=green] at (9) {$\nu=\bar{N}EE\bar{N}E\bar{N}E\bar{N}$};

\node [right=1.1cm,] at ($(7)!0.5!(8)$) {$\longmapsto$};
\node [right=1cm] at ($(7)!0.8!(8)$) {$\Vie$};

\node [left=1cm,] at ($(1)!0.5!(4)$) {$\longmapsto$};
\node [left=1.5cm,] at ($(1)!0.8!(4)$) {$\varphi$};
\node [left=2cm,] at ($(1)!0.5!(4)$) {$\pi=143295768$};
\begin{scope}[scale=0.7, xshift=50mm, yshift=-50mm]
\draw [green] (0,0) -- (0,1) -- (1,1) -- (2,1) -- (2,2) -- (3,2) -- (3,3) -- (4,3) -- (4,4);
\foreach \x/\y in {0/0, 0/1, 1/1, 2/1, 2/2, 3/2, 3/3, 4/3, 4/4} {
    \filldraw [black] (\x,\y) circle (2pt);
}  

    \begin{scope}[scale=1.05, xshift=-2.5mm, yshift=1.5mm]
           \draw [orange] (0,0) -- (0,1) -- (1,1) -- (2,1) -- (2,2) -- (2,3) -- (3,3) -- (3,4) -- (4,4);
        \foreach \x/\y in {0/0, 0/1, 1/1, 2/1, 2/2, 2/3, 3/3, 3/4, 4/4} {
        \filldraw [purple] (\x,\y) circle (2pt);
       }
    \end{scope}
\end{scope}

\end{tikzpicture}
\caption{An example of the bijection  $\Theta=\Vie\circ\varphi$. \label{vie}}
\end{figure}
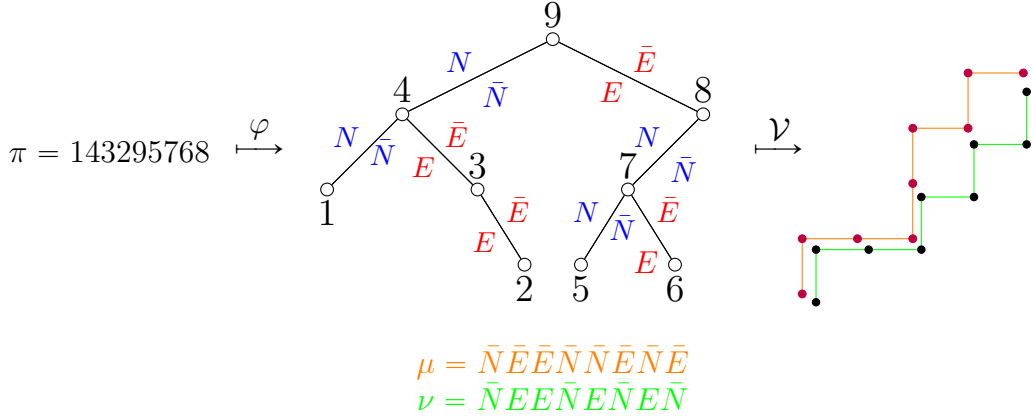

Next we recall the construction of $\Vie$.  For a binary tree $T \in \B_n$, traverse $T$ in preorder and record the edge bilaterally  as follows: label  $N$ (resp.,~{$\bar{N}$}) when traversing a left edge for the first (resp.,~{second}) time, and label $E$ (resp.,~{$\bar{E}$}) when traversing a right edge for the first (resp.,~{second}) time. 
Associate a pair of words $(\mu,\nu)$ through this traversal,  where $\mu$ keeps track only of the two letters $\{\bar{N},\bar{E}\}$ and $\nu$ keeps track only on $\{\bar{N},E\}$. See Fig.~\ref{vie} for an example where $(\mu,\nu)=(\bar{N} \bar{E} \bar{E} \bar{N} \bar{N} \bar{E} \bar{N} \bar{E}, \bar{N}EE\bar{N}E\bar{N}E\bar{N})$. 
Define $\Vie(T)$ to be the pair of walks in $\PW_n$, where the top  walk is encoded by $\mu$, while the bottom walk is encoded by $\nu$: 
\begin{itemize}
\item  the $i$-th step of top walk is east step $\iff$ $\mu_i=\bar{E}$; 
\item the $i$-th step of bottom walk is east step $\iff$ $\nu_i=E$.
\end{itemize}
It is convenience to write $\Vie(T)=(\mu,\nu)$. 
Pr\'{e}ville-Ratalle and Viennot~\cite{Vie} showed that $\Vie: \B_n\rightarrow\PW_n$ is a bijection. 
Now define $\Theta=\Vie\circ\varphi$. For an example of $\Theta$, see Fig.~\ref{vie}.

\begin{proof}[{\bf Proof of Theorem~\ref{thm:theta}}]
To complete  the proof of Theorem~\ref{thm:theta}, it remains to show~\eqref{da:hv}.
For  $\pi\in \S_n(231)$, assume that its greatest letter decomposition is $(\a,n,\b)$ with $\a\in\S_k(231)$ and $\b\in\S_{n-k-1}(231)$. Suppose that $\Theta(\pi)=(\mu,\nu)$, $\Theta(\a)=(\mu_L,\nu_L)$ and $\Theta(\b)=(\mu_R,\nu_R)$. We proceed to show~\eqref{da:hv} by induction on $n$ and distinguish the following three cases.
\begin{itemize}
\item If $\b$ is empty, then
$$
\DES(\pi)=\DES(\a)\quad\text{and}\quad\ASC(\pi^{-1})=\ASC(\alpha^{-1})\cup\{n-1\}. 
$$
As $\varphi(\pi)$ is constructed by attaching $\varphi(\a)$ as left branch to a new node, we have 
$$
\nu=\nu_L\cdot\bar{N}\quad\text{and}\quad\mu=\mu_L\cdot\bar{N}. 
$$
Thus,~\eqref{da:hv} holds by induction  in this case.
\item  If $\a$ is empty, then
$$
\DES(\pi)=\{1\}\cup\{i+1: i\in\DES(\b)\}\quad\text{and}\quad\ASC(\pi^{-1})=\ASC(\b^{-1}). 
$$
As $\varphi(\pi)$ is constructed by attaching $\varphi(\a)$ as right branch to a new node, we have 
$$
\nu=E\cdot\nu_R\quad\text{and}\quad\mu=\mu_R\cdot\bar{E}. 
$$
Thus,~\eqref{da:hv} holds by induction in this case.
\item If both $\a$ and $\b$ are non-empty, then 
$$
\DES(\pi)=\DES(\a)\cup\{k+1\}\cup\{i+k+1: i\in\DES(\b)\} 
$$
and 
$$
\ASC(\pi^{-1})=\ASC(\a^{-1})\cup\{k\}\cup\{i+k:i\in\ASC(\b^{-1})\}.
$$
In this case, $\varphi(\pi)$ is  the binary tree whose root has $\varphi(\a)$ as left branch and  $\varphi(\b)$ as right branch. Then,
$$
\nu=\nu_L\cdot\bar{N}E\cdot\nu_R\quad\text{and}\quad\mu=\mu_L\cdot\bar{N}\cdot\mu_R\cdot\bar{E}.
$$
As $\H(\nu)=\{i\in[n-1]:\nu_i=E\}$ and $\V(\mu)=\{i\in[n-1]:\mu_i=\bar{E}\}$, the desired property~\eqref{da:hv} then follows by induction.
\end{itemize}
As~\eqref{da:hv} is true in the above three cases, the proof of the theorem is complete. 
\end{proof}

A permutation $\pi\in\S_n$ is a {\em Baxter permutation} if it avoids the {\em vincular patterns} (see~\cite{Kit2}) $2\underline{41}3$ and $3\underline{14}2$, i.e., there are no indices $1\leq i<j<j+1<k\leq n$ such that  
$$
\pi_{j+1}<\pi_i<\pi_k<\pi_j \quad\text{or}\quad \pi_{j}<\pi_k<\pi_i<\pi_{j+1}.
$$
Denote by $\Bax_n$ the set of all Baxter permutations in $\S_n$. The following observation is obvious. 
\begin{lemma}\label{lem:Bax}
All stack-sortable permutations are Baxter permutations, i.e., $\S_n(231)\subseteq\Bax_n$ for $n\geq1$. 
\end{lemma} 

A conjectured bijection, which has been proved recently by the first two authors~\cite{Lin3}, relating three descent-based statistics on Baxter permutations to east steps of non-intersecting triples of walks was proposed by Dilks.

\begin{definition}
For a permutation $\pi\in\S_n$, define the following two set-valued extensions of the descent statistic as:
\begin{itemize}
          \item $\widetilde\DT(\pi)=\{\pi_{i}-1:\pi_{i}>\pi_{i+1}\}\subseteq [n-1]$, the set of all {\bf\em modified descent tops} of $\pi$;
          \item $\DB(\pi)=\{\pi_{i+1}:\pi_{i}>\pi_{i+1}\}$, the set of all {\bf\em descent bottoms} of $\pi$.
\end{itemize}
For example, if $\pi=143295768$, then $\widetilde\DT(\pi)=\{2,3,6,8\}$ and $\DB(\pi)=\{2,3,5,6\}$.
\end{definition}
For a subset $S\subseteq[n-1]$, a walk of size $n$ is said to {\em encode $S$} if for any $1\leq i\leq n-1$, 
$$
i\in S\Leftrightarrow  \text{the $i$-th step of this walk is $E$}. 
$$
Given a Baxter permutation $\pi\in\Bax_n$, define $\Gamma(\pi)$ to be the triple of non-intersecting walks of size $n$, where 
\begin{itemize}
\item the bottom walk encodes $\DB(\pi^{-1})$;
\item the middle walk encodes $\DES(\pi)$;
\item and the top walk encodes $\widetilde{\DT}(\pi^{-1})$. 
\end{itemize}
\begin{theorem}[Lin and Liu~\cite{Lin3}]\label{thm:Lin3}
The map $\Gamma$ is a bijection between $\Bax_n$ and triples of non-intersecting walks of size $n$.
\end{theorem}

\begin{lemma}\label{lem:DB}
For any $\pi\in\S_n(231)$, $\DB(\pi^{-1})=\DES(\pi)$. 
\end{lemma}

\begin{proof}
Note that $i\in\DB(\pi^{-1})$ iff the letter $\pi_i-1$ appears to the right of $\pi_i$ in $\pi$, i.e., $\pi_i-1=\pi_j$ for some $j>i$. As $\pi$ is $231$-avoiding, we have $\pi_k<\pi_i$ for any $i<k\leq j$, which forces $i\in\DES(\pi)$ whenever $i\in\DB(\pi^{-1})$. The lemma then follows from the fact that $|\DB(\pi^{-1})|=\des(\pi^{-1})=\des(\pi)=|\DES(\pi)|$ for any $\pi\in\S_n(231)$. 
\end{proof}

By Lemmas~\ref{lem:Bax} and~\ref{lem:DB}, the bijection $\Gamma$ restricts  to a bijection between $\S_n(231)$ and $\PW_n$ (as the bottom walk coincides with the middle walk). Therefore, the following  result is a consequence of Theorem~\ref{thm:theta}. 

\begin{corollary}
The composition $\Gamma^{-1}\circ\Theta: \pi\mapsto \sigma$ is a bijection on $\S_n(231)$ that sends the pair $(\DES(\pi),\DES(\pi^{-1}))$ to $(\DES(\sigma),\widetilde\DT(\sigma^{-1}))$. 
\end{corollary}

\subsection{Krattenthaler's bijection from $321$-avoiding permutations to Dyck paths}
\label{sec:Krat}
A {\em Dyck path} of order $n$ is a lattice path in $\mathbb{N}^2$ from $(0,0)$ to $(n,n)$, consists of  {\em north} step $N=(1,0)$ and {\em east} step $E=(0,1)$, which does not pass above the line $y=x$.  Denote by $\D_n$  the set of all Dyck paths of order $n$. 
Given a Dyck path $D$, a {\em platform} of $D$ is a maximal sequence of consecutive east steps. Let $\PT(D)$ be the multiset of the lengths of all platforms of $D$. For the Dyck path $D$ in Fig.~\ref{dyck path}, we have 
$\PT(D)=\{4,2,1^3\}$.

We first recall a classical bijection  $\tau$ between $\B_n$ and $\D_n$ (see~\cite{DV}). Given  $T\in\B_n$, consider the {\em completion of $T$}, denoted  $\widetilde{T}$, which is a complete binary tree obtained from $T$ by adding $n+1$ edges so that all the initial nodes of $T$ become internal nodes. Then $\tau(T)$ is constructed by recording the steps when   $\widetilde{T}$ is traversed  in preorder: we record an east step when visiting a left edge and a north step when visiting a right edge; see Fig.~\ref{dyck path} for an example of $\tau$. The following property of $\tau$ is clear from the construction. 

\begin{lemma}\label{lem:PT}
For each $T\in\B_n$, $\LC(T)=\PT(\tau(T))$. 
\end{lemma}

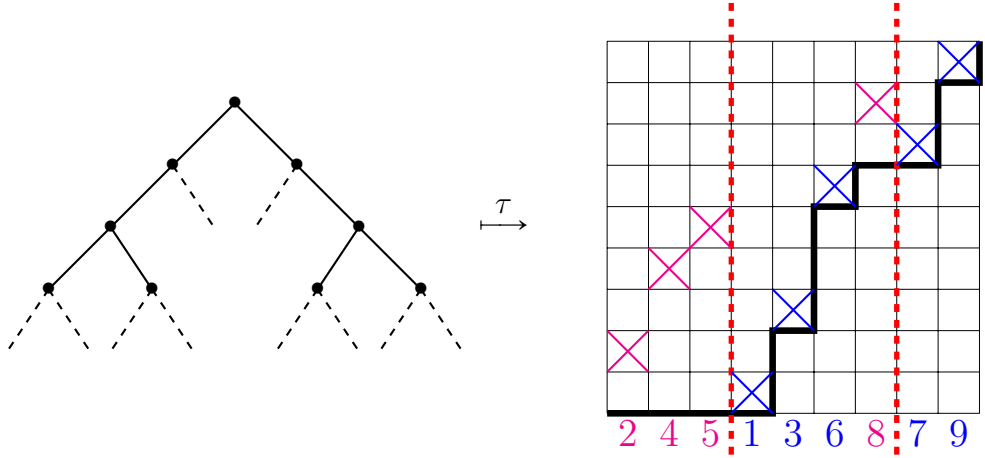
\begin{figure}
\begin{center}
\begin{tikzpicture}[scale=0.55]
\node at (-9,7.5){$\bullet$};\draw[thick,-] (-9,7.5) to (-4.5,3);
\node at (-7.5,6){$\bullet$};\draw[thick,dashed] (-7.5,6) to (-8.5,4.5);
\node at (-6,4.5){$\bullet$};\draw[thick,-] (-6,4.5) to (-7,3);
\node at (-7,3){$\bullet$};\draw[thick,dashed] (-7,3) to (-8,1.5);\draw[thick,dashed] (-7,3) to (-6,1.5);
\node at (-4.5,3){$\bullet$};\draw[thick,dashed] (-4.5,3) to (-5.5,1.5);
\draw[thick,dashed] (-4.5,3) to (-3.5,1.5);
\draw[thick,-] (-9,7.5) to (-13.5,3);
\node at (-10.5,6){$\bullet$};\draw[thick,dashed] (-10.5,6) to (-9.5,4.5);
\node at (-12,4.5){$\bullet$};\draw[thick,-] (-12,4.5) to (-11,3);
\node at (-11,3){$\bullet$};\draw[thick,dashed] (-11,3) to (-10,1.5);\draw[thick,dashed] (-11,3) to (-12,1.5);
\node at (-13.5,3){$\bullet$};\draw[thick,dashed] (-13.5,3) to (-12.5,1.5);
\draw[thick,dashed] (-13.5,3) to (-14.5,1.5);
\node at (-2.5,4.5){$\longmapsto$};
\node at (-2.5,5){$\tau$};
\draw[step=1cm,black,very thin] (0,0) grid (9,9);
\draw [line width=2.5pt,black] (0,0)--(4,0)--(4,2)--(5,2)--(5,5)--(6,5)--(6,6)--(8,6)--(8,8)--(9,8)--(9,9);
\draw [line width=2pt, red, dashed] (3,-1)--(3,10) (7,-1)--(7,10);
\draw(0.5,-0.5) node[scale=1.25,text=magenta]{$2$};
\draw(1.5,-0.5) node[scale=1.25,text=magenta]{$4$};
\draw(2.5,-0.5) node[scale=1.25,text=magenta]{$5$};
\draw(3.5,-0.5) node[scale=1.25,text=blue]{$1$};
\draw(4.5,-0.5) node[scale=1.25,text=blue]{$3$};
\draw(5.5,-0.5) node[scale=1.25,text=blue]{$6$};
\draw(6.5,-0.5) node[scale=1.25,text=magenta]{$8$};
\draw(7.5,-0.5) node[scale=1.25,text=blue]{$7$};
\draw(8.5,-0.5) node[scale=1.25,text=blue]{$9$};
\draw [thick, magenta] (0,2)--(1,1) (1,4)--(2,3) (2,5)--(3,4) (6,8)--(7,7);
\draw [ thick, magenta] (0,1)--(1,2) (1,3)--(2,4) (2,4)--(3,5) (6,7)--(7,8);
\draw [ thick, blue] (3,0)--(4,1) (4,2)--(5,3) (5,5)--(6,6) (7,6)--(8,7) (8,8)--(9,9);
\draw [ thick, blue] (3,1)--(4,0) (4,3)--(5,2) (5,6)--(6,5) (7,7)--(8,6) (8,9)--(9,8);
\end{tikzpicture}
\end{center}
\caption{An example of  $\psi\circ\tau$. \label{dyck path}}
\end{figure}

Next we recall Krattenthaler's bijection~\cite{Kra} $\psi$ between $\S_n(321)$ and $\D_n$. Given a permutation $\pi\in\S_n$, a position $i\in[n-1]$ is an {\em excedance}  of $\pi$ if $\pi_i >i$. It is known that a permutation is $321$-avoiding if and only if both the subsequence formed by letters in excedance positions and the one formed by the remaining letters are increasing. For $\pi\in\S_n(321)$, we represent $\pi$ on the $n\times n$ grid with crosses on the squares $(i,\pi_i)$; see Fig.~\ref{dyck path} for the representation of $\pi=245136879$. With this representation,  $\psi(\pi)$ is the Dyck path from the lower-left corner to the uper-right corner of the grid leaving all the crosses to the left and remaining always as close to the diagonal as possible. See  Fig.~\ref{dyck path} for an example of $\psi$. A platform of a Dyck path is {\em long} if it has length at least $2$.  The following properties  of $\psi$  can  be easily checked. 

\begin{lemma}\label{lem:Kra}
Let $\pi\in\S_n(321)$ and $D=\psi(\pi)$. Then,
\begin{itemize}
\item[(a)] $i$ is a non-excedance (i.e.,~$\pi_i>i$) of $\pi$  iff the $i$-th east step of $D$ is the last step of a platform. 
\item[(b)] $i$ is a descent of $\pi$ iff the $i$-th east step of $D$ is the penultimate step of a long platform. 
\end{itemize}
\end{lemma}

Lemma~\ref{lem:Kra}~(b)  was first observed in~\cite{Lin}, which is crucial  in the proof of Theorem~\ref{thm:gf}~(i).
For a permutation $\pi\in\S_n$, denote by $\mne(\pi)$ the  {\em maximum number of non-overlapping excedances} of $\pi$. For example, if $\pi=245136879$, then $\mne(\pi)=3$. Setting $y=1$ in~\eqref{X:Y} we have  
\begin{equation}\label{X:mnd}
\sum_{\pi\in\S_n(231)}x^{\mnd(\pi)}=\sum_{T\in\B_n} x^{\X(T)}.
\end{equation} 
The following equdistribution is a consequence of~\eqref{X:mnd}, Lemma~\ref{lem:PT} and Lemma~\ref{lem:Kra}~(a). 

\begin{proposition}
For $n\geq1$, we have 
\begin{equation}\label{equ:mne}
\sum_{\pi\in\S_n(231)}x^{\mnd(\pi)}=\sum_{\pi\in\S_n(321)}x^{\mne(\pi)}. 
\end{equation}
\end{proposition}

In view of~\eqref{X:Y} and~\eqref{equ:mne},  one may wonder whether there exists a natural statistic ``$\st$'' on $321$-avoiding permutations so that 
\begin{equation*}
\sum_{\pi\in\S_n(231)}x^{\mnd(\pi)}y^{\mna(\pi^{-1})}=\sum_{\pi\in\S_n(321)}x^{\mne(\pi)}y^{\st(\pi)}. 
\end{equation*}
We shall answer this question in next subsection (see Theorem~\ref{thm:Phi2}). 
\subsection{A bijection from $321$-avoiding permutations to pairs of non-intersecting walks}
Given a permutation $\pi\in\S_n$, a position $i$, $2\leq i\leq n$, is a {\em weak excedance\footnote{For the sake of convenience, we exclude position $1$ from weak excedances.}}  of $\pi$ if $\pi_i\geq i$. Let  $\mnw(\pi)$ be the  {\em maximum number of non-overlapping weak excedances} of $\pi$. We introduce the set of {\em modified weak excedances} of $\pi$ as 
$$
\widetilde\WEXC(\pi):=\{i-1: \text{ $i$ is a weak excedance of $\pi$}\}.
$$
Denote by $\EXC(\pi)$ the set of all excedances of $\pi$.  For example, if $\pi=134275968$, then $\EXC(\pi)=\{2,3,5,7\}$, $\widetilde\WEXC(\pi^{-1})=\{1,4,5,7\}$ and $\mnw(\pi^{-1})=3$. 

\begin{theorem}\label{thm:Phi2}
There is a bijection $\Phi: \pi\mapsto(\mu,\nu)$ between $\S_n(321)$ and $\PW_n$ such that 
 \begin{equation}\label{ew:hv}
 (\EXC(\pi),\widetilde\WEXC(\pi^{-1}))=(\H(\nu),\V(\mu)).
 \end{equation}
 Consequently,  $\Phi^{-1}\circ\Theta: \pi\mapsto\sigma$ is a bijection between $\S_n(231)$ and $\S_n(321)$ such that 
 $$
(\DES(\pi),\ASC(\pi^{-1}))=(\EXC(\sigma),\widetilde\WEXC(\sigma^{-1})).
$$
In particular, 
\begin{equation*}\label{mne:mnw}
\sum_{\pi\in\S_n(231)}x^{\mnd(\pi)}y^{\mna(\pi^{-1})}=\sum_{\pi\in\S_n(321)}x^{\mne(\pi)}y^{\mnw(\pi^{-1})}
\end{equation*}
and there follows another permutation interpretation for the Catalan--Schett polynomials 
$$
C_n(x,y)=\sum_{\pi\in\S_n(321)}x^{n-2\mne(\pi)}y^{n-2\mnw(\pi^{-1})}.
$$
\end{theorem}

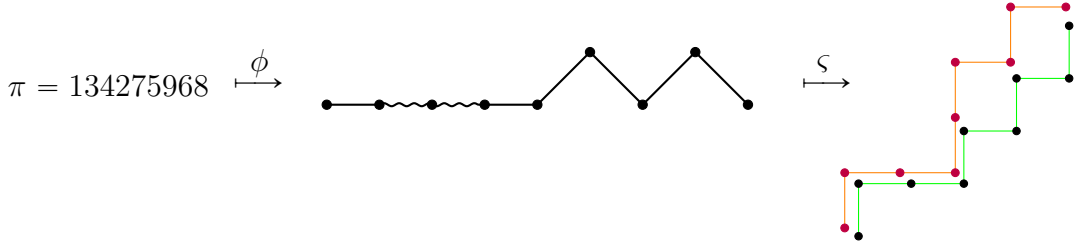
\begin{figure}
\centering
\begin{tikzpicture}
\begin{scope}[scale=0.7, xshift=-41mm, yshift=-25mm]
\draw [thick,black] (-1,0)--(0,0) ; \draw [thick,black] (2,0)--(3,0) -- (4,1) -- (5,0) -- (6,1) -- (7,0);

\draw[thick,decorate, decoration={snake, amplitude=0.2mm, segment length=2mm}] (0,0) -- (2,0);

\foreach \x/\y in {-1/0, 0/0, 1/0, 2/0, 3/0, 4/1, 5/0, 6/1, 7/0} {
    \filldraw [black] (\x,\y) circle (2.5pt);
}  
\end{scope}
\node [right=1.1cm,] at ($(7)!0.5!(8)$) {$\longmapsto$};
\node [right=1cm] at ($(7)!0.8!(8)$) {$\varsigma$};

\node [left=1.5cm,] at ($(1)!0.5!(4)$) {$\longmapsto$};
\node [left=2cm,] at ($(1)!0.8!(4)$) {$\phi$};
\node [left=2.5cm,] at ($(1)!0.5!(4)$) {$\pi=134275968$};
\begin{scope}[scale=0.7, xshift=50mm, yshift=-50mm]
\draw [green] (0,0) -- (0,1) -- (1,1) -- (2,1) -- (2,2) -- (3,2) -- (3,3) -- (4,3) -- (4,4);
\foreach \x/\y in {0/0, 0/1, 1/1, 2/1, 2/2, 3/2, 3/3, 4/3, 4/4} {
    \filldraw [black] (\x,\y) circle (2pt);
}  

    \begin{scope}[scale=1.05, xshift=-2.5mm, yshift=1.5mm]
           \draw [orange] (0,0) -- (0,1) -- (1,1) -- (2,1) -- (2,2) -- (2,3) -- (3,3) -- (3,4) -- (4,4);
        \foreach \x/\y in {0/0, 0/1, 1/1, 2/1, 2/2, 2/3, 3/3, 3/4, 4/4} {
        \filldraw [purple] (\x,\y) circle (2pt);
       }
    \end{scope}
\end{scope}

\end{tikzpicture}
\caption{An example of the bijection  $\Phi=\varsigma\circ\phi$. \label{vie2}}
\end{figure}

The bijection $\Phi$ consists of two main steps, the first of which is a bijection due to Lin and Fu~\cite{LF} between
$321$-avoiding permutations and two-colored Motzkin paths. 

Recall that a {\em Motzkin path} of length $n$ is a lattice path in $\N^2$ starting at $(0,0)$, ending at $(n,0)$, with three possible steps: 
$$
\text{$(1,1)=U$ (up step), $(1,-1)=D$ (down step)\,\,\,and\,\,\,$(1,0)=H$ (horizontal step).}
$$
A {\em two-colored Motzkin path} is  a Motzkin path whose  horizontal steps are colored by $\red{H}$ or $\blue{\tilde H}$.  See Fig.~\ref{vie2} for a display of the two-colored Motzkin path $\red{H}\blue{\tilde{H}\tilde{H}}\red{H}UDUD$, where $\tilde H$'s are drawn as wavy lines. Denote by  $\Mo_n^{(2)}$ the set of all two-colored Motzkin paths of length $n$.

In order to define Lin and Fu's bijection $\phi:\S_n(321)\rightarrow\Mo_{n-1}^{(2)}$, we need two  vectors to keep track of the values and positions of excedances.
 For $\pi\in\S_n(321)$, let 
$$
\val(\pi)=(v_1,\ldots,v_n)\quad\text{and}\quad\pos(\pi)=(p_1,\ldots,p_n),
$$
where $v_i=\chi(i>\pi^{-1}_i)$ and $p_i=\chi(\pi_i>i)$. For instance, for  $\pi=134275968\in\S_{9}(321)$, we have
$$\val(\pi)=(0,0,1,1,0,0,1,0,1),\quad\pos(\pi)=(0,1,1,0,1,0,1,0,0).$$
Now, define $\phi(\pi)=M=m_1m_2\ldots m_{n-1}\in\Mo_{n-1}^{(2)}$ as
\begin{align*}
m_{i}=
\begin{cases}
\,\,U\quad&\text{if $v_{i+1}=0$ and $p_i=1$},\\
\,\,\blue{\tilde H}\quad&\text{if $v_{i+1}=p_i=1$},\\
\,\,D\quad&\text{if $v_{i+1}=1$ and $p_i=0$},\\
\,\,\red{H}\quad&\text{if $v_{i+1}=p_i=0$}.
\end{cases}
\end{align*}
Continuing with our  example, we have $\phi(134275968)=\red{H}\blue{\tilde{H}\tilde{H}}\red{H}UDUD$. Lin and Fu~\cite{LF} proved that $\phi$ is a bijection, which forms the first step of $\Phi$.

The second step of $\Phi$ is the bijection $\varsigma: M=m_1m_2\ldots m_{n-1}\mapsto (\mu,\nu)$  defined as
\begin{itemize}
\item if $m_i=U$, then the $i$-th step of $\mu$ is $N$ and the $i$-th step of $\nu$ is $E$;
 \item if $m_i=D$, then the $i$-th step of $\mu$ is $E$ and the $i$-th step of $\nu$ is $N$;
\item if $m_i=\red{H}$, then both the $i$-th step of $\mu$ and $\nu$ are $N$; 
\item if $m_i=\blue{\tilde H}$, then both the $i$-th step of $\mu$ and $\nu$ are $E$.
\end{itemize}
See Fig.~\ref{vie2} for an example of $\varsigma$. It is clear that $\varsigma$ is a bijection between $\Mo_{n-1}^{(2)}$ and $\PW_n$. 

\begin{proof}[{\bf Proof of Theorem~\ref{thm:Phi2}}] Set  $\Phi=\varsigma\circ\phi$. From the constructions of the two steps, it can be easily checked   that the bijection $\Phi$ satisfies the desired property~\eqref{ew:hv}. The second statement  then follows from  Theorem~\ref{thm:theta} and Theorem~\ref{thm:sch},  which completes the proof. 
\end{proof}

\subsection{The bijection $\Psi: \S_n(321)\rightarrow \S_n(231)$}
The bijection $\Psi$ consists of two main steps, one is the Simion--Schmidt bijection~\cite{SS} (see also~\cite[p.~148]{Kit2}) between $\S_n(321)$ and $\S_n(312)$, and another one is a newly constructed  bijection between $\S_n(312)$ and $\S_n(231)$ basing on  the Foata--Zeilberger bijection~\cite[Sec.~8]{FZ} (see also~\cite[Sec.~4]{Cor}).

For the sake of completeness, we review the Simion--Schmidt bijection  $\eta: \S_n(321)\rightarrow\S_n(312)$ that preserves the values and positions of left peaks. For each $\pi\in\S_n(321)$, the permutation $\sigma=\eta(\pi)\in\S_n(312)$ can be constructed by the following algorithm:
\begin{itemize}
\item[(1)] (Start) set $\sigma_1=\pi_1$; $x\leftarrow \pi_1$, $i\leftarrow1$;
\item[(2)] $i\leftarrow i+1$; if $\pi_i<x$, then set $\sigma_i=\max\{k:  k<x, k\neq \sigma_j\text{ for all $j<i$}\}$; otherwise,  $\pi_i>x$ and set $\sigma_i=\pi_i$, $x\leftarrow \pi_i$; do step (2) again until $i=n$. 
\end{itemize}
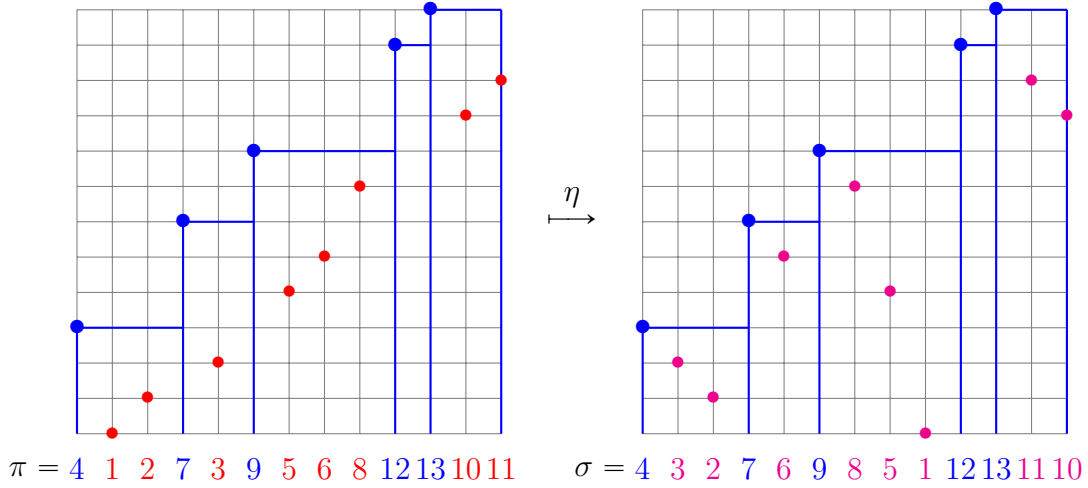
\begin{figure}
\begin{center}
\begin{tikzpicture}[scale=0.47]
\draw[step=1cm,thin,gray] (0,0) grid (12,12);\draw[step=1cm, thin,gray] (16,0) grid (28,12);
\blue{\draw[ thick,-] (0,0) to (0,3);\draw[ thick,-] (12,0) to (12,12);
\draw[ thick,-] (0,3) to (3,3);\draw[ thick,-] (3,0) to (3,6);\draw[ thick,-] (3,6) to (5,6);
\draw[thick,-] (5,0) to (5,8);\draw[ thick,-] (5,8) to (9,8);\draw[ thick,-] (9,0) to (9,11);
\draw[ thick,-] (9,11) to (10,11);\draw[ thick,-] (10,0) to (10,12);\draw[thick,-] (10,12) to (12,12);}
\node at (0,3){\blue{\large{$\bullet$}}};\node at (14,6){$\longmapsto$};\node at (14,6.7){$\eta$};
\node at (0,-1){\blue{$4$}};
\node at (3,6){\blue{\large{$\bullet$}}};
\node at (3,-1){\blue{$7$}};
\node at (5,8){\blue{\large{$\bullet$}}};
\node at (5,-1){\blue{$9$}};
\node at (9,11){\blue{\large{$\bullet$}}};
\node at (9,-1){\blue{$12$}};
\node at (10,12){\blue{\large{$\bullet$}}};
\node at (10,-1){\blue{$13$}};
\red{\node at (1,0){\red{$\bullet$}};
\node at (1,-1){$1$};
\node at (2,1){\red{$\bullet$}};
\node at (2,-1){$2$};
\node at (4,2){\red{$\bullet$}};
\node at (4,-1){$3$};
\node at (6,4){\red{$\bullet$}};
\node at (6,-1){$5$};
\node at (7,5){\red{$\bullet$}};
\node at (7,-1){$6$};
\node at (8,7){\red{$\bullet$}};
\node at (8,-1){$8$};
\node at (11,9){\red{$\bullet$}};
\node at (11,-1){$10$};
\node at (12,10){\red{$\bullet$}};
\node at (12,-1){$11$};}

\blue{\draw[ thick,-] (16,0) to (16,3);\draw[ thick,-] (28,0) to (28,12);
\draw[ thick,-] (16,3) to (19,3);\draw[ thick,-] (19,0) to (19,6);\draw[ thick,-] (19,6) to (21,6);
\draw[thick,-] (21,0) to (21,8);\draw[ thick,-] (21,8) to (25,8);\draw[ thick,-] (25,0) to (25,11);
\draw[ thick,-] (25,11) to (26,11);\draw[ thick,-] (26,0) to (26,12);\draw[thick,-] (26,12) to (28,12);}
\node at (16,3){\blue{\large{$\bullet$}}};
\node at (16,-1){\blue{$4$}};
\node at (19,6){\blue{\large{$\bullet$}}};
\node at (19,-1){\blue{$7$}};
\node at (21,8){\blue{\large{$\bullet$}}};
\node at (21,-1){\blue{$9$}};
\node at (25,11){\blue{\large{$\bullet$}}};
\node at (25,-1){\blue{$12$}};
\node at (26,12){\blue{\large{$\bullet$}}};
\node at (26,-1){\blue{$13$}};
\magenta{\node at (17,2){\magenta{$\bullet$}};
\node at (17,-1){$3$};
\node at (18,1){\magenta{$\bullet$}};
\node at (18,-1){$2$};
\node at (20,5){\magenta{$\bullet$}};
\node at (20,-1){$6$};
\node at (22,7){\magenta{$\bullet$}};
\node at (22,-1){$8$};
\node at (23,4){\magenta{$\bullet$}};
\node at (23,-1){$5$};
\node at (24,0){\magenta{$\bullet$}};
\node at (24,-1){$1$};
\node at (27,10){\magenta{$\bullet$}};
\node at (27,-1){$11$};
\node at (28,9){\magenta{$\bullet$}};
\node at (28,-1){$10$};}
\node at (-1.2,-1){$\pi=$};
\node at (14.8,-1){$\sigma=$};
\end{tikzpicture}
\end{center}
\caption{An example of  $\eta$. \label{fig:eta}}
\end{figure}
See Fig.~\ref{fig:eta} for an example of $\eta$. The inverse algorithm for constructing $\pi=\eta^{-1}(\sigma)\in\S_n(321)$ from $\sigma\in\S_n(312)$ reads as follows:
\begin{itemize}
\item[(1*)] (Start) set $\pi_1=\sigma_1$; $x\leftarrow \sigma_1$, $i\leftarrow 1$;
\item[(2*)] $i\leftarrow i+1$; if $\sigma_i<x$, then set $\pi_i=\min\{k: k<x, k\neq \pi_j\text{ for all $j<i$}\}$; otherwise,  $\sigma_i>x$ and set $\pi_i=\sigma_i$, $x\leftarrow \sigma_i$; do step (2*) again until $i=n$. 
\end{itemize}
The following lemma is clear from the above algorithms.

\begin{lemma}\label{lem:eta}
The bijection $\eta: \S_n(321)\rightarrow\S_n(312)$ preserves the values and positions of left peaks. 
\end{lemma}

For the second step of $\Psi$, we need to review  the Foata--Zeilberger bijection $\Psi_{FZ}$, which is a variant of the classical Fran\c{c}on--Viennot bijection~\cite{FV} (see also~\cite{Lin3}), from $\S_n$ to the restricted Laguerre histories of length $n$. Recall that a {\em restricted Laguerre history} of length $n$ is a pair $(M,w)$, where $M=m_1m_2\cdots m_n\in\Mo_n^{(2)}$ and $w=w_1w_2\cdots w_n\in\N^n$ is a weight function satisfying 
$$
0\leq w_i\leq \begin{cases}
h_i(M),\quad&\text{if $m_i=U, \red{H}$};\\
h_i(M)-1, \quad&\text{if $m_i=D, \blue{\tilde H}$}.
\end{cases}
$$
Here $h_i(M)$ is the {\em height} of the $i$-th step of $M$: 
$$
h_i(M):=|\{j\mid j<i, m_j=U\}|-|\{j\mid j<i, m_j=D\}|.
$$
Let $\L_n$ denote the set of all restricted Laguerre histories of length $n$. 

Given a permutation $\pi\in\S_n$, we set
   $\pi_0=0$ and $\pi_{n+1}=+\infty$.
The  bijection $\Psi_{FZ}: \S_n\rightarrow\L_{n}$ can be defined as $\Psi_{FZ}(\pi)=(M,w)$, where for  $i\in[n]$ with $\pi_j=i$: 
$$
  m_i=\left\{
  \begin{array}{ll}
  U&\quad\mbox{if $\pi_{j-1}>\pi_j<\pi_{j+1}$},  \\
 D&\quad\mbox{if $\pi_{j-1}<\pi_j>\pi_{j+1}$},  \\
  \red{H} &\quad\mbox{if $\pi_{j-1}<\pi_j<\pi_{j+1}$},\\
  \blue{\tilde{H}}&\quad\mbox{if $\pi_{j-1}>\pi_j>\pi_{j+1}$},
  \end{array}
  \right.
  $$
  and $w_i$ is the number of $\underline{31}2$-patterns with $i$ representing the $2$, i.e., 
$$w_i=(\underline{31}2)_i(\pi):=|\{k: k<j\text{ and } \pi_k<\pi_j=i<\pi_{k-1}\}|.$$
For example, if $\pi=524139768\in\S_9$, then $\Psi_{FZ}(\pi)=(UU\red{H}DDU\blue{\tilde H}\red{H}D,002100010)$.

The inverse algorithm $\Psi_{FZ}^{-1}$ building a permutation $\pi$ (in $n$ steps) from a Laguerre history $(M,w)\in\L_{n}$ can be described iteratively as:
\begin{itemize}
\item Initialization: $\pi=\diamond$;
\item At the $i$-th ($1\leq i\leq n$) step of the algorithm, replace the $(w_i+1)$-th $\diamond$ (from left to right) of $\pi$ by
$$
\begin{cases}
\,\diamond i\diamond&\quad \text{if $m_i=U$},\\
\, i\diamond&\quad \text{if $m_i=\red{H}$},\\
 \, i& \quad\text{if $m_i=D$},\\
 \, \diamond i&\quad \text{if $m_i=\blue{\tilde H}$};
\end{cases}
$$
\item The final permutation is obtained by removing  the last remaining $\diamond$. 
\end{itemize}
For example, if $(M,w)=(UU\red{H}DDU\blue{\tilde H}\red{H}D,002100010)\in\L_9$, then
\begin{align*}
\pi&=\diamond\rightarrow \diamond1\diamond\rightarrow\diamond2\diamond1\diamond\rightarrow \diamond2\diamond13\diamond\rightarrow \diamond2413\diamond\rightarrow 52413\diamond\\
&\quad\rightarrow 52413\diamond6\diamond\rightarrow 52413\diamond76\diamond\rightarrow 52413\diamond768\diamond\rightarrow 524139768. 
\end{align*}

From the inverse algorithm $\Psi_{FZ}^{-1}$,  one can check easily the following lemma. 

\begin{lemma}[Corteel~\cite{Cor}]\label{lem:cort}
Suppose that $(M,w)\in\L_n$ and $\pi=\Psi_{FZ}^{-1}(M,w)$. Then for any $1\leq i\leq n$, 
$$
(\underline{31}2)_i(\pi)+(2\underline{31})_i(\pi)=
\begin{cases}
h_i(M), &\quad\text{if $m_i=U, \red{H}$};\\
h_i(M)-1, &\quad\text{if $m_i=D, \blue{\tilde H}$}.
\end{cases}
$$
Here $(2\underline{31})_i(\pi)$ denotes the number of $2\underline{31}$-patterns in $\pi$ with $i$ representing the $2$, i.e.,
$$
(2\underline{31})_i(\pi)=|\{k: k>j\text{ and } \pi_{k+1}<\pi_j=i<\pi_{k}\}|.
$$
\end{lemma}
Note that $\pi$ is $231$-avoiding (resp.,~$312$-avoiding) iff $(2\underline{31})_i(\pi)=0$ (resp.,~$(\underline{31}2)_i(\pi)$)  for all $i$. Thus, for any $\pi\in\S_n(312)$, $\Psi_{FZ}(\pi)=(M,w)$ with $w_i=0$ for all $i$.  Lemma~\ref{lem:cort} then induces the bijection $\psi_{FZ}:\pi\mapsto\sigma$ between $\S_n(312)$ and $\S_n(231)$, where 
$\sigma=\Psi_{FZ}^{-1}(M,h)$ with 
$$
h_i=\begin{cases}
h_i(M),\quad&\text{if $m_i=U, \red{H}$};\\
h_i(M)-1, \quad&\text{if $m_i=D, \blue{\tilde H}$}.
\end{cases}
$$
For example, if $\pi=432768951\in\S_n(312)$, then $\Psi_{FZ}(\pi)=(M,w)$ with $w_i=0$ for all $i$ and $M=UU\blue{\tilde H}D\blue{\tilde H}UD\red{H}D$. Thus, $\sigma=\psi_{FZ}(\pi)=\Psi_{FZ}^{-1}(M,011101110)$ is constructed iteratively as:
\begin{align*}
\sigma&=\dia\rightarrow\dia1\dia\rightarrow\dia1\dia2\dia\rightarrow\dia1\dia32\dia\rightarrow\dia1432\dia\rightarrow\dia51432\dia\\
&\quad\rightarrow \dia51432\dia6\dia\rightarrow\dia5143276\dia\rightarrow\dia51432768\dia\rightarrow951432768. 
\end{align*}
\begin{lemma}\label{lem:FZ}
The bijection $\psi_{FZ}:\S_n(312)\rightarrow\S_n(231)$ preserves the values of left peaks. 
\end{lemma}
\begin{proof}
Note that for any $\pi\in\S_n$, if  $\Psi_{FZ}(\pi)=(M,w)$, then 
$$\LPK(\pi)=\{2\leq i\leq n: m_i=D\}.$$
The result then follows from the construction of $\psi_{FZ}$. 
\end{proof}

In view of Lemmas~\ref{lem:eta} and~\ref{lem:FZ}, we set $\Psi=\psi_{FZ}\circ\eta$, which is a bijection between $\S_n(321)$ and $\S_n(231)$ preserving the statistic ``$\LPK$''.

\subsection{The bijection $\vartheta$ from plane trees to stack-sortable permutations.}
We aim to construct the bijection $\vartheta$ for Theorem~\ref{bij:pt-231}. For this purpose, we need a decomposition of plane trees parallel to the first letter decomposition of stack-sortable permutations.

Given a tree $T\in \T_n$, let $r$ be the root of $T$. We decompose $T$ into two smaller trees $T_L$ and $T_R$ according to the following two cases.
\begin{itemize}
\item[(1)] The root $r$ of $T$ has leaves as its children. Let $w$ denote the leftmost leaf-child of $r$ in tree $T$. Then define  $T_L$ (resp.,~{$T_R$})  the subtree of $T$ rooted at $r$ on the left (resp.,~right) of the edge $(r,w)$; see  Fig.~\ref{tree1}. Note that all children of the root of $T_L$ are non-leaves. 
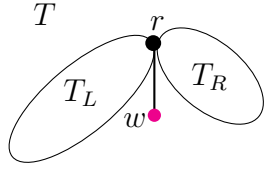
\begin{figure}[h]
\centering
\begin{tikzpicture}[scale=0.3]
\node at (9,23) {$T$};
\node at (13.8,21.7) {\circle*{6}};\node at (14,22.6) {$r$};
\node at (13,18.3) {$w$};
\node at (16.3,20.2) {$T_{R}$};
\node at (10.6,19.5) {$T_{L}$};
\draw [rotate=50] (26,0.5) ellipse (1.5 and 2.8);
\draw [rotate=-50] (-7.9,20.5) ellipse (1.5 and 4);
\draw[thick] (13.85,22) to  (13.85,18.5);
\node at (13.85,18.5) {\magenta{\circle*{5}}};
\end{tikzpicture}
\caption{The decomposition when the root $r$ has leaves as its children.\label{tree1}}
\end{figure}
\item[(2)] All children of the root $r$ are non-leaves. Let $x$ be the  first leaf  under preorder traversal of $T$ and let $y$ be the parent of $x$.
The nodes between $r$ and $y$ (under the preorder) are denoted by $a,b,\ldots,c$ (if any). Denote by $T_H,T_A,T_B,\ldots,T_C,T_F$ the subtrees rooted at $r,a,b,\ldots,c,y$ that are on the right of the path from $r$ to $x$, respectively (see Fig.~\ref{tree2} for the illustration). 

Now $T_L$ is constructed by attaching $T_H$ and $T_F$ at $y$ of the edge $(y,x)$ from left side and right side respectively. Note that $y$ is the root of $T_L$ and $x$ is the leftmost leaf-child of $y$. The tree $T_R$ is obtained by hanging $T_A$ at $a$, $T_B$ at $b$, $\ldots$, $T_C$ at $c$ to a new node in this specified order. See Fig.~\ref{tree2} for an illustration of constructing $T_L$ and $T_R$ from $T$. 
\begin{figure}[h]
\centering
\begin{tikzpicture}[scale=0.3]
\node at (-2,1.7) {\circle*{6}};
\node at (-4.6,-1.5) {\circle*{5}};
\node at (-7.1,-4.6) {\circle*{5}};

\node at (-13.5,-12.4) {\circle*{5}};
\node at (-9.15,-7.1){\circle*{5}};

\draw [thick] (-2,1.7) to  (-7.1,-4.6);
\draw [dashed] (-7.1,-4.6) to  (-9.15,-7.1);
\draw [thick] (-9.15,-7.1) to  (-10.9,-9.2);
\draw [thick]   (-10.9,-9.2) to (-13.5,-12.4);
\draw [rotate=50] (0,0) ellipse (1 and 2.5);
\draw [rotate=50] (-4,0) ellipse (1 and 2.5);
\draw [rotate=50] (-8,0) ellipse (1 and 2.5);
\draw [rotate=50] (-11.5,0) ellipse (1 and 2.5);
\draw [rotate=50] (-15,0) ellipse (1 and 2.5);

\node at (-11.5,-10) {\magenta{\circle*{6}}};

\node at (0,0) {$T_{H}$};
\node at (-2,1.7) {\circle*{6}};

\node at (-9.15,0){$T$};
\node at (-2.8,-3) {$T_A$};
\node at (-5.6,-6) {$T_B$};
\node at (-7.6,-8.4) {$T_C$};
\node at (-10,-11.4) {$T_F$};
\node [left] at (-2,1.7){$r$};
\node [left] at (-4.6,-1.5) {$a$};
\node [left] at (-7.1,-4.6) {$b$};

\node[left] at (-13.5,-12.4) {$x$};
\node [left] at (-11.5,-10) {$y$};
\node [left] at (-9.15,-7.1){$c$};
\Large{\node at (3.5,-8) {$\longrightarrow$};}

\begin{scope}[xshift=80mm, yshift=-20mm][scale=0.3]
\node [above, font=\small] at (2.8,-2.9) {$y$};

\node at (2.8,-6){\circle*{5}};
\node [below, font=\small] at (2.8,-6) {$x$};
\node [font=\small]at  (5,-4.3) {$T_F$};
\node [font=\small]at (1,-4.3) {$T_H$};

\draw [rotate=50] (-0.1,-6.5) ellipse (1 and 2.5);
\draw [rotate=-50] (3.95,-2.45) ellipse (1 and 2.5);
\draw [thick] (2.8,-2.9)--(2.8,-6);
\Large{\node at (9,-6) {$+$};}
\node at (2.8,-2.9){\magenta{\circle*{6}}};

\node at (17,-1) {\circle*{5}};
\node at (13.5,-3.5) {\circle*{4.5}};
\draw [thick] (17,-1)--(13.5,-3.5);
\draw [thick] (17,-1)--(15.7,-3.3);
\node at (15.7,-3.5) {\circle*{4.5}};
\node at (19.8,-3.5){\circle*{4.5}};
\draw [thick] (17,-1)--(19.8,-3.5);

\draw [rotate=-20] (14,-1.15) ellipse (1 and 2.5);
\draw[xshift=18mm,yshift=-24mm,rotate=-10] (13.9,-1.2) ellipse (1 and 2.5);
\draw[xshift=30mm,yshift=-23mm, rotate=15] (16,-8) ellipse (1 and 2.5);

\node [font=\small] at (12.8,-5.6) {$T_A$};
\node [font=\small] at (15.5,-5.6){$T_B$};
\node [font=\small] at (18,-5.6){$\dots$};
\node [font=\small] at (20.5,-5.6){$T_C$};
\node [left,font=\small] at (19.9,-3.5){$c$};
\node [left, font=\small] at (14,-3) {$a$};
\node [left, font=\small]  at (16,-3.2){$b$};
\node [font=\small] at (0,1) {$T_L$};
\node [font=\small] at (15,1) {$T_R$};
\end{scope}

\end{tikzpicture}
\caption{The decomposition when all children of the root $r$ are non-leaves.\label{tree2}}
\end{figure}
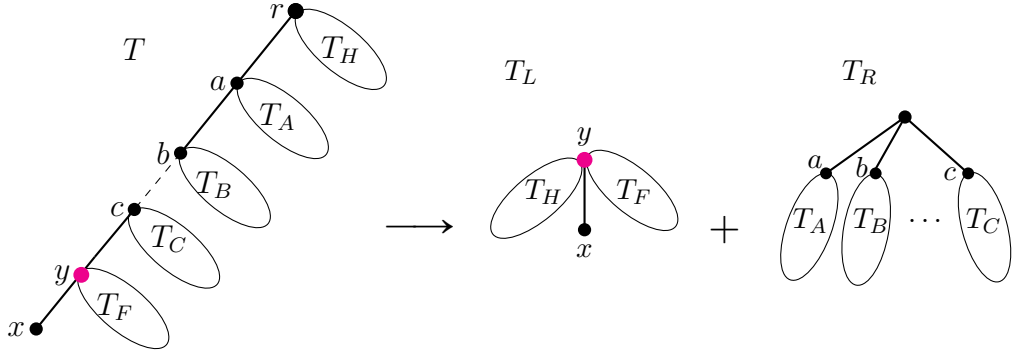
\end{itemize}
With the above decomposition of $T$ into $(T_L,T_R)$, we define
 the map $\vartheta$ recursively as:
\begin{equation*}
\vartheta(T) = (k\cdot\vartheta(T_L))\oplus\vartheta(T_R),
\end{equation*}
where $k$ is the number of nodes of $T_L$. 
The following property of $\vartheta$ can be proved easily   by induction on $n$. 
\begin{lemma}\label{idr:eo}
Let $T\in\T_n$ and $\pi=\vartheta(T)$. If the root of $T$ has leaves as its children, then $\idr(\pi)$ is odd. Otherwise, $\idr(\pi)$ is even. 
\end{lemma}

\begin{proof}[{\bf Proof of Theorem~\ref{bij:pt-231}}]
With Lemma~\ref{idr:eo}, we see that $\vartheta$ is a bijection by induction on $n$.
It remains to show that 
\begin{equation}\label{mk=mnd}
\mk(T)=\mnd(\pi)
\end{equation}
for any $T\in\T_n$ and $\pi=\vartheta(T)$.

We proceed by induction on $n$. Suppose that $\pi'=\vartheta(T_L)\in\S_{k-1}$ and $\pi''=\vartheta(T_R)\in\S_{n-k}$. Then  $\pi=(k\cdot\pi')\oplus\pi''$. We consider the following two cases:
 \begin{itemize}
\item If the root of $T$ has a leaf-child, then by the decomposition in Fig.~\ref{tree1}, we have 
\begin{align*}
\mk(T)&=\mk(T_L)+\mk(T_R)\\
&=\mnd(\pi')+\mnd(\pi'')=\mnd(\pi),
\end{align*}
where the second equality follows by induction and the last equality follows from Lemma~\ref{idr:eo} (as $\idr(\pi')$ is even). 
\item If all children of the root of $T$ are non-leaves, then by decomposition in Fig.~\ref{tree2}, we have 
\begin{align*}
\mk(T)&=\mk(T_L)+\mk(T_R)+1\\
&=\mnd(\pi')+\mnd(\pi'')+1=\mnd(\pi),
\end{align*}
where the second equality follows by induction and the last equality follows  from Lemma~\ref{idr:eo} (as $\idr(\pi')$ is odd). 
\end{itemize}
By induction, we have proved that~\eqref{mk=mnd} holds for all $T\in\T_n$, which completes the proof.
\end{proof}

In view of Theorem~\ref{bij:pt-231}, we pose the following interesting open problem for further research.

\begin{problem} Find a natural statistic ``$\st$'' on plane trees so that
\begin{equation*}
\sum_{\pi\in\S_n(231)}x^{\mnd(\pi)}y^{\mna(\pi^{-1})}=\sum_{T\in\T_n}x^{\mk(T)}y^{\st(T)}. 
\end{equation*}
\end{problem}

\section{Generating functions}
\label{sec:gf}
This section is devoted to the proof of Theorem~\ref{thm:gf}. As mentioned earlier, Krattenthaler's bijection $\psi$ defined in Sec.~\ref{sec:Krat} and the bijection $\Psi$ constructed in Theorem~\ref{bij:LPK} will be useful. 

\subsection{Generating function for the statistic pair $(\oar,\ear)$ over $\S_n(321)$} Recall that $\oar(\pi)$ (resp.,~$\ear(\pi)$) counts the number of odd (resp.,~even) ascending  runs in $\pi$. 
By Lemma~\ref{lem:Kra}~(b),  under Krattenthaler's bijection $\psi$, the positions of descents of $\pi\in\S_n(321)$ are transformed to the  positions of the  penultimate steps of long platforms of the Dyck path $D=\psi(\pi)$. 
 Thus, we partition the Dyck path $D$ into segments by breaking $D$ before the last step of each long  platform, and let  $\Comp(D)$  be the composition recording the number of east steps in each segment successively. For $D$ in Fig.~\ref{dyck path}, we have $\Comp(D)=(3,4,2)$.
Introduce $$\OP(D):= \# \text{ odd parts of  }\Comp(D)\quad\text{and}\quad\EP(D):= \# \text{ even parts of }\Comp(D).
$$
Then,  
$$
\oar(\pi)=\OP(D)   \text{\quad and \quad } \ear(\pi)=\EP(D).
$$
Thus, we have
\begin{equation*}
G(t,x,y)=\sum_{n\geq1}t^n\sum_{\pi\in\S_n(321)}x^{\oar(\pi)}y^{\ear(\pi)}=\sum_{n\geq1}t^n\sum_{D\in\D_n}x^{\OP(D)}y^{\EP(D)}.
\end{equation*}
Basing on this interpretation of $G(t,x,y)$ as Dyck paths, we are going to prove Theorem~\ref{thm:gf}~(i) by using the first return decomposition of Dyck paths. 

In order to achieve our goal, we need to partition Dyck paths into four classes according to the parities of the initial and the terminal parts of their recording compositions. For a given $D\in\D_n$, let $\i(D)$ and $\t(D)$ be respectively the {\em initial} and the {\em terminal parts} in $\Comp(D)$.
We abbreviate  $a\equiv b \pmod{2}$ as $a\equiv b$. Introduce  the following four refined generating functions:
\begin{align*}
EE(t,x,y)&=\sum_{n\geq1}t^n\sum_{D\in\EE_n}x^{\OP(D)}y^{\EP(D)}\quad\text{with $\EE_n:=\{D\in\D_n: \i(D)\equiv\t(D)\equiv0\}$},\\
EO(t,x,y)&=\sum_{n\geq1}t^n\sum_{D\in\EO_n}x^{\OP(D)}y^{\EP(D)}\quad\text{with $\EO_n:=\{D\in\D_n: \i(D)\equiv0,\t(D)\equiv1\}$},\\
OE(t,x,y)&=\sum_{n\geq1}t^n\sum_{D\in\OE_n}x^{\OP(D)}y^{\EP(D)}\quad\text{with $\OE_n:=\{D\in\D_n: \i(D)\equiv1,\t(D)\equiv0\}$},\\
OO(t,x,y)&=\sum_{n\geq1}t^n\sum_{D\in\OO_n}x^{\OP(D)}y^{\EP(D)}\quad\text{with $\OO_n:=\{D\in\D_n: \i(D)\equiv\t(D)\equiv1\}$}.
\end{align*}
We set $\EE$ (resp.,~$\EO,\OE,\OO$) to be the union of all $\EE_n$ (resp.,~$\EO_n,\OE_n,\OO_n$). 
A Dyck path without any long platform is called {\em zigzag}. Note that there is only one zigzag Dyck path with given fixed length. 

First we notice the following equality. 
\begin{lemma}\label{eo=oe}
We have
  \begin{equation}\label{EO=OE}
      EO(t,x,y) = OE(t,x,y).
   \end{equation}
\end{lemma}

\begin{proof}
This follows from the bijection $\psi: \S_n(321)\rightarrow\D_n$ and the involution $\pi\mapsto\pi^{rc}$ on $\S_n(321)$, where $\pi^{rc}:=(n+1-\pi_n)(n+1-\pi_{n-1})\dots(n+1-\pi_1)$. 
\end{proof}

\begin{remark}
It seems not easy to prove~\eqref{EO=OE} bijectively in terms of Dyck paths directly. 
\end{remark}
Next we prove the following functional equations for $EE(t,x,y)$ and $EO(t,x,y)$. 
\begin{lemma}\label{pro2}
The generating function $EE=EE(t,x,y)$ and $EO=EO(t,x,y)$ satisfy the following functional equations:
\begin{equation}\label{eq:ee}
\begin{aligned}
EE=&\overbrace{\frac{y t^2}{1-t^2}}^{(1.1)}
+\overbrace{\frac{yt\cdot OE}{x}}^{(1.2)}
+\overbrace{\frac{yt}{x}\Big(OE-\frac{tx}{y}\big(EE-\frac{y t^2}{1-t^2}\big)\Big)\Big( 1+\frac{EE}{y}+\frac{OE}{y}\Big)}^{(2.1)}\\
&+\overbrace{\frac{yt}{x}
\Big(OO-\frac{tx}{1-t^2}-\frac{tx}{y}\cdot EO \Big)
\Big(\frac{EE}{y}-\frac{t^2}{1-t^2}+\frac{y\cdot OE }{x^2}+\frac{ty}{x(1-t^2) }\Big)}^{(2.2)}
\end{aligned}
\end{equation}
and
\begin{equation}\label{eq:eo}
\begin{aligned}
EO=&\overbrace{\frac{yt}{x}\Big(OO-\frac{tx}{1-t^2}\Big)}^{(1)}
+\overbrace{\frac{t}{x}\Big(OE-\frac{tx}{y}\big(EE-\frac{t^2 y}{1-t^2}\big)\Big)\big(EO+EE\big)}^{(2.1)}\\
&+\overbrace{\frac{yt}{x}\Big(OO-\frac{tx}{y}\cdot EO-\frac{tx}{1-t^2}\Big)\Big(\frac{1}{1-t^2}+\frac{EO}{y}+\frac{y\cdot OO}{x^2}-\frac{ty}{x(1-t^2)}\Big)}^{(2.2)}.
\end{aligned}
\end{equation}
\end{lemma}

\begin{proof}
Consider the first return decomposition for Dyck paths. Every nonempty Dyck path $D$ can be decomposed uniquely according to the position of first return to $y=x$ as:
  $$D=ED_LND_R,$$
where $D_L$  and $D_R$ are smaller Dyck paths that are possibly empty. 
Given $D\in\EE$ with the above first return decomposition, we distinguish the following two cases.
\begin{itemize}
\item[(1)] If $D_L$ is empty (see Fig.~\ref{ee}~(a)), then we further distinguish two subcases.
   \begin{itemize}
      \item[(1.1)] When $D_R \in \OO$ and is zigzag, we know $D$ is also zigzag. This case contributes $\frac{y t^2}{1-t^2}$, which is the summand term $(1.1)$ in~\eqref{eq:ee}.
     \item[(1.2)] Otherwise, $D_R \in \OE$ and we have
                 $$\OP(D)=\OP(D_R)-1 \text{  and } \EP(D)= \EP(D_R)+1. $$
                 This case contributes $\frac{yt\cdot OE}{x}$, which is the summand term $(1.2)$ in~\eqref{eq:ee}.
   \end{itemize}
   \begin{figure}
\centering
$\begin{array}{cc}
\begin{tikzpicture}[scale=0.5]
\draw (0,0) --(1,0)--(1,1);
\draw[cyan,line width=1.5pt] (1,1)--(5,1)--(5,5)--(6,5)--(6,6);
\draw[black] (0,0) -- (6,6);
\draw[red,dotted,line width=1.5pt](4,-1)--(4,8);
\foreach \y in {1,...,5} {
    \fill (5,\y) circle (3pt); 
}
\foreach \y in {5,...,6} {
    \fill (6,\y) circle (3pt); 
}
\foreach \x in {1,...,4} {
    \fill (\x,1) circle (3pt);}
\foreach \x in {0,...,1} {
    \fill (\x,0) circle (3pt);}
\draw[fill=red] (1,1) circle [radius=4pt];
\draw[->={Latex[length=8mm]}, red, thick] (-0.5,1) -- (0.5,1) node[midway, above, sloped] {\resizebox{3cm}{!}{\parbox{3cm}{first return}}};

\node[above,text=cyan] at (3,3.5) {$D_R$};

\begin{scope}[xshift=80mm]
\draw (0,0) --(1,0) (5,4)--(5,5);
\draw[blue,line width=1.5pt] (1,0)--(3,0) -- (3,2) -- (5,2) -- (5,4);
\draw[cyan,line width=1.5pt] (5,5)--(6,5)--(6,6)--(7,6)--(7,7)--(9,7)--(9,9)--(10,9)--(10,10);
\foreach \x in {0,1,...,3} {
    \fill (\x,0) circle (3pt);
}
\foreach \x in {3,...,4} {
    \fill (\x,2) circle (3pt);
}
\foreach \x in {6,...,6} {
    \fill (\x,5) circle (3pt);
}
\foreach \y in {0,...,2} {
    \fill (3,\y) circle (3pt); 
}
\foreach \y in {2,...,5} {
    \fill (5,\y) circle (3pt); 
}

\foreach \x in {6,...,6} {
    \fill (\x,6) circle (3pt);}
\foreach \x in {6,...,7} {
    \fill (\x,6) circle (3pt);
}
\foreach \x in {7,...,9} {
    \fill (\x,7) circle (3pt);
}
\foreach \y in {7,...,9} {
    \fill (9,\y) circle (3pt); 
}
\foreach \y in {9,...,10} {
    \fill (10,\y) circle (3pt); 
}
\draw[fill=red] (5,5) circle [radius=4pt];

\draw[black,line width=1.25pt] (0,0) -- (5,5);
\draw[black,line width=1.25pt] (5,5) -- (10,10);
\draw[blue, line width=1.5pt,dashed](1,0)--(5,4);
\draw[red,dotted,line width=1.5pt](2,-1)--(2,8) (4,-1)--(4,8) 
(8,2)--(8,10);
\draw[->={Latex[length=8mm]}, red, thick] (3.5,5) -- (4.5,5) node[midway, above, sloped] {\resizebox{3cm}{!}{\parbox{3cm}{first return}}};
\node[above,text=cyan] at (6.5,7) {$D_R$};
\node[above,text=blue] at (5,0) {$D_L$};
\draw[decorate,decoration={brace,amplitude=10pt,mirror}] (1,-0.1) -- (3,-0.1) node[midway,below,yshift=-10pt] {$\geq 2$};

\end{scope}
\node at (2,-0.5) [below right] {$(a)$};
\node at (14,-0.5) [below right] {$(b)$};

\end{tikzpicture}
\end{array}$
\caption{Two cases for Dyck paths $D$ with $\i(D)$ even.\label{ee}}
\end{figure}
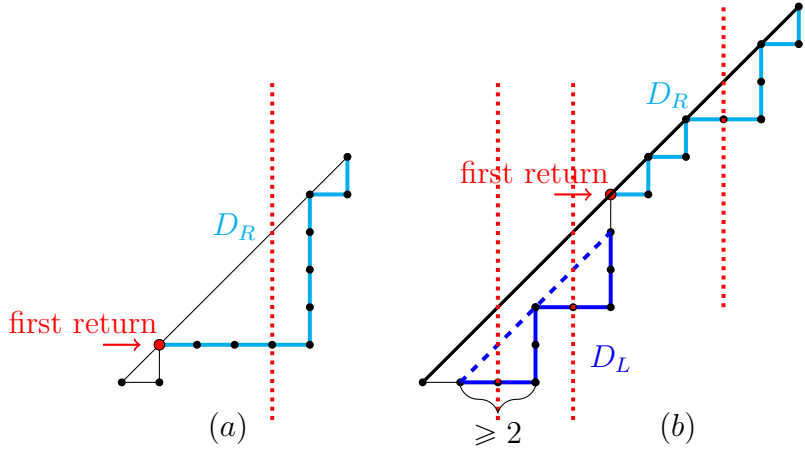
\item[(2)] If $D_L$ is nonempty (see Fig.~\ref{ee}~(b)), then we  further consider two subcases.
    \begin{itemize}
     \item[(2.1)]  If $D_L \in \OE$ and $D_L$ can not begin with $EN$ (otherwise, $\i(D)=1$, contradicting with $D
     \in\EE$),  then either  $\t(D_R)$  is even or  $D_R$ is empty. In such case, we have 
     $$\OP(D)=\OP(D_L)+\OP(D_R)-1$$ and 
     $$ \EP(D)=\left\{
             \begin{array}{ll}
                  \EP(D_L)+1&\mbox{if $D_R$ is empty},  \\
                  \EP(D_L)+\EP(D_R)&\mbox{if $\t(D_R)$  is even.}
            \end{array}
       \right.$$
Note that  the generating function for all $D_L\in\OE$ that begin with $EN$ is $\frac{tx}{y}(EE-\frac{yt^2}{1-t^2})$. Thus, this case contributes 
$$
\frac{yt}{x}\Big(OE-\frac{tx}{y}\big(EE-\frac{y t^2}{1-t^2}\big)\Big)\Big( 1+\frac{EE}{y}+\frac{OE}{y}\Big),
$$
which is the summand term $(2.1)$ in~\eqref{eq:ee}.
    
    \item[(2.2)] If $D_L \in \OO$ and $D_L$ can not begin with $EN$, then either $\t(D_R)$ is even or $D_R$ is zigzag of odd length.  We consider the following three cases.
    \begin{itemize}
         \item[a)] If $D_R \in \OE$,  then
          \begin{align*}
          &\OP(D)=\OP(D_L)+\OP(D_R)-3,\\
          &\EP(D)=\EP(D_L)+\OP(D_R)+2. 
          \end{align*}  
         \item[b)] If $D_R \in\EE$ and is not zigzag, then
         \begin{align*}
         &\OP(D)=\OP(D_L)+\OP(D_R)-1, \\
         &\EP(D)=\EP(D_L)+\OP(D_R).
         \end{align*}
         \item[c)] If $D_R$ is zigzag of odd length, then
         \begin{align*}
         & \OP(D)=\OP(D_L)+\OP(D_R)-3,\\
         &\EP(D)=\EP(D_L)+\OP(D_R).
        \end{align*}
        \end{itemize}
        Note that  the generating function for all $D_L\in\OO$ that begin with $EN$ is $\frac{tx}{1-t^2}+\frac{tx}{y}\cdot EO$. Thus, this case contributes 
        $$
        \frac{yt}{x}
\Big(OO-\frac{tx}{1-t^2}-\frac{tx}{y}\cdot EO \Big)
\Big(\frac{EE}{y}-\frac{t^2}{1-t^2}+\frac{y\cdot OE }{x^2}+\frac{ty}{x(1-t^2) }\Big),
        $$
        which is the summand term $(2.2)$ in~\eqref{eq:ee}.
    \end{itemize}
\end{itemize}
 Summing over all the above cases gives~\eqref{eq:ee}.
  
By utilizing similar discussions as above, we get Eq.~\eqref{eq:eo} for $EO(t,x,y)$.
\end{proof}

Finally, we prove the following functional equation for $OO(t,x,y)$, which is more involved. 
 
\begin{lemma}\label{pro3}
The generating function $OO=OO(t,x,y)$ satisfies the functional equation
\begin{equation}\label{eq:o}
\begin{aligned}
OO=&\frac{tx}{1-t^2}
+\frac{tx}{y}\cdot EO
+\frac{t^2x^2}{y}\Big(EO+\frac{y}{1-t^2}\Big)+t^2y\Big(OO-\frac{tx}{1-t^2}\Big)\\
&+\frac{t^2x^2}{y}\Big(EO+\frac{t^2y}{1-t^2}\Big)\Big(\frac{EO}{y}+\frac{1}{1-t^2}+\frac{y\cdot OO}{x^2}-\frac{ty}{x(1-t^2)}\Big)\\
&+\frac{t^2x^2}{y^2}\Big(EE-\frac{t^2y}{1-t^2}\Big)\big(EO+OO\big)\\
&+t^2y\Big(OO-\frac{tx}{1-t^2}\Big)\Big(\frac{EO}{y}+\frac{1}{1-t^2}+\frac{y\cdot OO}{x^2}-\frac{ty}{x(1-t^2)}\Big)\\
&+t^2\Big(OE+\frac{tx}{1-t^2}\Big)\big(EO+OO\big)\\
&+\frac{tx}{y^2}\Big(EE-\frac{t^2y}{1-t^2}-\frac{ty}{x}\cdot OE\Big)\big(OO+EO\big)\\
&+\frac{tx}{y}\Big(EO-\frac{ty}{x}\big(OO-\frac{tx}{1-t^2}\big)\Big)\Big(\frac{EO}{y}+\frac{1}{1-t^2}+\frac{y \cdot OO}{x^2}-\frac{ty}{x(1-t^2)}\Big).
\end{aligned}
\end{equation}
\end{lemma}

\begin{proof}
Consider the first return decomposition for $D\in\OO$:
$$D=ED_LND_R,$$
where $D_L$  and $D_R$ are smaller Dyck paths that are possibly empty.
We shall discuss the following three cases.
\begin{itemize}
\item[(1)] If $D_L$ is empty, then $D_R$ have the following two possible conditions:
   \begin{itemize}
      \item[(1.1)] If $D_R$ is zigzag (or empty) of even length, then $D$ is  zigzag of odd length.  This case contributes $\frac{tx}{1-t^2}$ to~\eqref{eq:o}.
      \item[(1.2)] Otherwise, $D_R \in \EO$ and we have
                 $$\OP(D)=\OP(D_R)+1 \text{  and } \EP(D)= \EP(D_R)-1.$$
                 Thus, this case contributes $\frac{tx}{y}\cdot EO$ to~\eqref{eq:o}. 
   \end{itemize}
   
   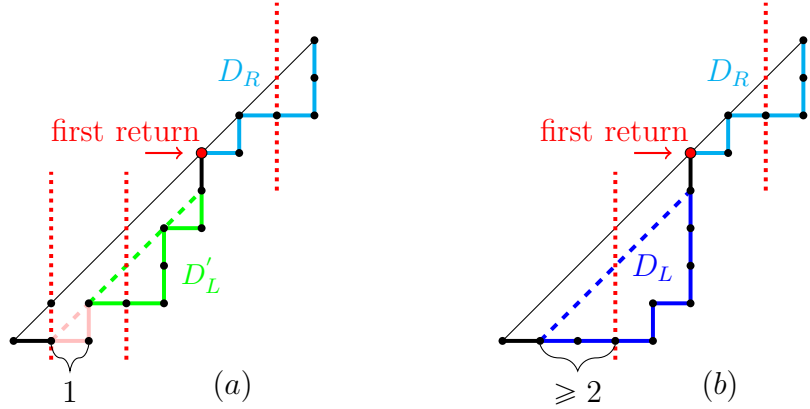
\begin{figure}
\centering
$\begin{array}{cc}
\begin{tikzpicture}[scale=0.5]
\draw[black,line width=1.5pt] (0,0) --(1,0) (5,4)--(5,5);
\draw[pink,line width=1.5pt] (1,0) --(2,0)--(2,1);
\draw[green,line width=1.5pt] (2,1)--(4,1)--(4,3)--(5,3)--(5,4);
\draw[cyan,line width=1.5pt] (5,5) --(6,5)--(6,6)--(8,6)--(8,8);
\draw[black] (0,0) -- (8,8);
\draw[pink, line width=1.5pt,dashed](1,0)--(2,1);
\draw[green, line width=1.5pt,dashed](2,1)--(5,4);
\draw[red,dotted,line width=1.5pt](1,-0.5)--(1,4.5) ;
\draw[red,dotted,line width=1.5pt](3,-0.5)--(3,4.5) ;
\draw[red,dotted,line width=1.5pt](7,4)--(7,9) ;
\foreach \y in {0,1,2} {
    \fill (\y,0) circle (3pt); 
}
\foreach \y in {2,...,4} {
    \fill (\y,1) circle (3pt); 
}
\foreach \x in {2,3} {
    \fill (4,\x) circle (3pt);
    } 
\foreach \x in {4,3} {
    \fill (5,\x) circle (3pt);
    }
\foreach \x in {6,7} {
    \fill (\x,\x-1) circle (3pt);
    }
\foreach \x in {6,8,} {
    \fill (\x,\x) circle (3pt);
    }   
\foreach \x in {6,7} {
    \fill (8,\x) circle (3pt);
    }
\draw[fill=red] (5,5) circle [radius=4pt];

\draw[->={Latex[length=8mm]}, red, thick] (3.5,5) -- (4.5,5) node[midway, above, sloped] {\resizebox{3cm}{!}{\parbox{3cm}{first return}}};

\node[above,text=cyan] at (6,6.5) {$D_R$};
\node[above,text=green] at (5,1) {\small{$D^{'}_L$}};
\draw[decorate,decoration={brace,amplitude=10pt,mirror}] (1,-0.1) -- (2,-0.1) node[midway,below,yshift=-10pt] {$ 1$};


\begin{scope}[xshift=130mm]
\draw[black,line width=1.5pt] (0,0) --(1,0) (5,4)--(5,5);
\draw[blue,line width=1.5pt] (1,0) --(4,0)--(4,1)--(5,1)--(5,4);
\draw[cyan,line width=1.5pt] (5,5) --(6,5)--(6,6)--(8,6)--(8,8);
\draw[black] (0,0) -- (8,8);
\draw[blue, line width=1.5pt,dashed](1,0)--(5,4);
\draw[red,dotted,line width=1.5pt](3,-0.5)--(3,4.5) ;
\draw[red,dotted,line width=1.5pt](7,4)--(7,9) ;
\foreach \y in {0,1,2,3,4} {
    \fill (\y,0) circle (3pt); 
}
\foreach \y in {4,5} {
    \fill (\y,1) circle (3pt); 
}
\foreach \x in {2,4,3} {
    \fill (5,\x) circle (3pt);
    }
\foreach \x in {6,7} {
    \fill (\x,\x-1) circle (3pt);
    }
\foreach \x in {6,8} {
    \fill (\x,\x) circle (3pt);
    }   
\foreach \x in {6,7} {
    \fill (8,\x) circle (3pt);
    }
\draw[fill=red] (5,5) circle [radius=4pt];

\draw[->={Latex[length=8mm]}, red, thick] (3.5,5) -- (4.5,5) node[midway, above, sloped] {\resizebox{3cm}{!}{\parbox{3cm}{first return}}};

\node[above,text=cyan] at (6,6.5) {$D_R$};
\node[above,text=blue] at (4,1.3) {$D_L$};
\draw[decorate,decoration={brace,amplitude=10pt,mirror}] (1,-0.1) -- (3,-0.1) node[midway,below,yshift=-10pt] {$\geq 2$};

\end{scope}
\node at (5,-0.5) [below right] {$(a)$};
\node at (18,-0.5) [below right] {$(b)$};

\end{tikzpicture}
\end{array}$
\caption{Two cases for Dyck paths in $\OO$.\label{oo}}
\end{figure}

\item[(2)]   If $D_L$ is nonempty and $\i(D)=1$, then we consider the first return decomposition of $D_L$ as: 
            $$ D_L = END'_L,$$
where $D_L'$ may be empty; see Fig.~\ref{oo}~$(a)$. In this case, we distinguish the following two cases of $D'_L$.
      \begin{itemize}
       \item[(2.1)]  If $D'_L$ is empty, then either $\t(D_R)$ is odd or $D_R$ is empty.  We further distinguish three cases:
       \begin{itemize}
       \item[a)] If $D_R$ is empty, then 
             $$\OP(D)=2 \quad\text{and}\quad \EP(D)=0.$$
      \item[b)]  If $D_R \in \EO$ or $D_R$ is zigzag of even length, then
             $$\OP(D)=\OP(D_R)+2 \quad \text{and}\quad
             \EP(D)= \EP(D_R)-1.
             $$
       \item[c)] If $D_R \in \OO$,  then $D_R$ can not be  zigzag, otherwise $\t(D)$ will be even. We have 
             $$\OP(D)=\OP(D_R) \quad\text{and}\quad\EP(D)=\EP(D_R)+1
             $$ 
             \end{itemize}
             In summary, case~(2.1) contributes  
             $$t^2x^2+\frac{t^2x^2}{y}\Big(EO+\frac{t^2y}{1-t^2}\Big)+t^2y\Big(OO-\frac{tx}{1-t^2}\Big)$$ to~\eqref{eq:o}.
     \item[(2.2)] If $D'_L$ is nonempty, then we further distinguish the following  four cases.  
         \begin{itemize}
          \item[a)] If $D'_L \in \EO$ or $D'_L$ is zigzag with even length, then
           we need to consider three cases of $D_R$: 
         \begin{itemize}
         \item[(i)] if $D_R$ is empty, then 
         $$\OP(D)=\OP(D_L')+2\quad\text{and}\quad\EP(D)=\EP(D_L')-1;$$
         \item[(ii)] if $D_R \in \EO$ or $D_R$ is zigzag with even length, then 
          \begin{align*}
            &\OP(D)=\OP(D_L')+\OP(D_R)+2, \\
            &\EP(D)=\EP(D_L)+\EP(D_R)-2; 
           \end{align*}
        \item[(iii)] if $D_R\in \OO$ and $D_R$ is not zigzag, then 
           \begin{align*}
             &\OP(D)= \OP(D_L')+\OP(D_R)\\
             &\EP(D)=\EP(D_L')+\EP(D_R).
           \end{align*}
       \end{itemize}
       Thus, this case contributes 
       $$
       \quad\qquad\frac{t^2x^2}{y}\Big(EO+\frac{t^2y}{1-t^2}\Big)\Big(1+\frac{EO}{y}+\frac{t^2}{1-t^2}+\frac{y\cdot OO}{x^2}-\frac{ty}{x(1-t^2)}\Big)
       $$
       to~\eqref{eq:o}.
        \item[b)] If $D'_L \in \EE$ and $D_L'$ is not zigzag of even length,  then $\t(D_R)$ is odd. So  $D_R \in \EO$ or $D_R \in \OO$ and  we have
              \begin{align*}
                &\OP(D)=\OP(D_L')+\OP(D_R)+2\\
                &\EP(D)=\EP(D_L')+\EP(D_R)-2.
              \end{align*}
Thus,  this case contributes 
 $$
 \frac{t^2x^2}{y^2}\Big(EE-\frac{t^2y}{1-t^2}\Big)\big(EO+OO\big)
 $$
         to~\eqref{eq:o}.
         \item[c)] If $D'_L \in\OO$ and $D'_L$ is not  zigzag of odd length, then after considering the same three cases of $D_R$ as in case~a), we conclude that this case contributes 
         $$
         \quad\qquad t^2y\Big(OO-\frac{tx}{1-t^2}\Big)\Big(1+\frac{EO}{y}+\frac{t^2}{1-t^2}+\frac{y\cdot OO}{x^2}-\frac{ty}{x(1-t^2)}\Big). 
              $$

          \item[d)] If $D'_L \in \OE$ or $D'_L$ is zigzag of odd length, then $\t(D_R)$ is odd. So  $D_R \in \EO$ or $D_R \in \OO$ and  we have
              \begin{align*}
                &\OP(D)=\OP(D_L')+\OP(D_R),\\
                &\EP(D)=\EP(D_L')+\EP(D_R).
              \end{align*}
 Thus, this case contributes 
 $$
t^2\Big(OE+\frac{tx}{1-t^2}\Big)\big(EO+OO\big)
 $$
         to~\eqref{eq:o}.
        \end{itemize}
      \end{itemize}  
\item[(3)] If $D_L$ is nonempty and $\i(D)\geq3$ (see Fig.~\ref{oo}~(b)), then after similar discussions  as in the proof of  case~(2) in Lemma~\ref{pro2}, we conclude that this case contributions 
\begin{align*}
&\frac{tx}{y^2}\Big(EE-\frac{t^2y}{1-t^2}-\frac{ty}{x}\cdot OE\Big)\big(OO+EO\big)\\
&+\frac{tx}{y}\Big(EO-\frac{ty}{x}\big(OO-\frac{tx}{1-t^2}\big)\Big)\Big(\frac{EO}{y}+\frac{1}{1-t^2}+\frac{y \cdot OO}{x^2}-\frac{ty}{x(1-t^2)}\Big)
\end{align*}
  to~\eqref{eq:o}.
\end{itemize}

Summing over all the above three cases yields~\eqref{eq:o}.
\end{proof}

We can now prove Theorem~\ref{thm:gf}~(i).
\begin{proof}[{\bf Proof of Theorem~\ref{thm:gf}~(i)}]
By Lemma~\ref{eo=oe}, the generating functon $G(t,x,y)$ satisfies the functional equation
\begin{equation}\label{eq:G}
G= EE + 2\cdot EO + OO.
\end{equation}
Utilizing the Gr\"obner basis method for solving the system of functional equations~\eqref{eq:ee}, \eqref{eq:eo}, \eqref{eq:o} and~\eqref{eq:G}, we can prove  that $G(t,x,y)$ satisfies the algebraic equation of degree $4$ in~\eqref{alg:gf1}. This can be done using Maple (version 2022 or higher) with the following codes:
$$
\begin{aligned}
&\text{{\tt with(Groebner)}};\\
&{\tt A:= [f_1, f_2,f_3,f_4]};\\
&\text{{\tt Basis(A,\,plex(EO,EE,OO,G))}};
\end{aligned}
$$
where $f_1=0, f_2=0, f_3=0$, and $f_4=0$ are the functional equations~\eqref{eq:ee}, \eqref{eq:eo}, \eqref{eq:o} and~\eqref{eq:G}, respectively. 
\end{proof}

\subsection{Generating function for the statistic pair $(\lpk_o,\lpk_e)$ over $\S_n(321)$}
Recall that 
$$
M(t,x,y)=\sum_{n\geq1}t^n\sum_{\pi\in\S_n(321)}x^{\lpk_e(\pi)}y^{\lpk_o(\pi)}.
$$
By~\eqref{eo:lpk}, we have 
$$
M(t,x,y)=\sum_{n\geq1}t^n\sum_{\pi\in\S_n(231)}x^{\lpk_e(\pi)}y^{\lpk_o(\pi)}.
$$
Basing on this interpretation of $M(t,x,y)$,  we shall  prove Theorem~\ref{thm:gf}~(ii) by using the first letter decomposition of stack-sortable permutations in Lemma~\ref{dec:231}.

Give $\pi\in\S_n$, define two statistics:
\begin{itemize}
\item $\pk_e(\pi):=|\{2\leq i<n: \pi_{i-1}<\pi_i>\pi_{i+1}\text{ and }\pi_i\text{ is even}\}|$, the number of {\em even peaks }of $\pi$;
\item $\pk_o(\pi):=|\{2\leq i<n: \pi_{i-1}<\pi_i>\pi_{i+1}\text{ and }\pi_i\text{ is odd}\}|$, the number of {\em odd peaks }of $\pi$.
\end{itemize}
Introduce four generating functions
\begin{align*}
LE(t,x,y) &= \sum_{n\geq1}t^{2n}\sum_{\pi\in\S_{2n}(231)}x^{\lpk_e(\pi)}y^{\lpk_o(\pi)},\\
LO(t,x,y)& = \sum_{n\geq0}t^{2n+1}\sum_{\pi\in\S_{2n+1}(231)}x^{\lpk_e(\pi)}y^{\lpk_o(\pi)},\\
E(t,x,y) &=  \sum_{n\geq1}t^{2n}\sum_{\pi\in\S_{2n}(231)}x^{\pk_e(\pi)}y^{\pk_o(\pi)},\\
O(t,x,y) &= \sum_{n\geq0}t^{2n+1}\sum_{\pi\in\S_{2n+1}(231)}x^{\pk_e(\pi)}y^{\pk_o(\pi)}.
\end{align*}
 
\begin{lemma} The generating function $M(t,x,y)$ satisfies the following system of functional equations
\begin{equation}\label{eq:LE}
\begin{cases}
\,\, M=LE+LO,\\
\,\,LE=tx\cdot O\cdot(1+LE)+ t\cdot \widetilde{LO}+ty\cdot E\cdot \widetilde{LO},\\
\,\,LO= tx\cdot O \cdot LO+t(1+\widetilde{LE})+ty\cdot E\cdot(1+\widetilde{LE}),\\
\,\,E= t\cdot O\cdot(1+LE)+t(1+E)\cdot\widetilde{LO},\\
\,\,O=  t\cdot O \cdot LO+ t(1+E)(1+\widetilde{LE}),
\end{cases}
\end{equation}
where $\widetilde{LE}(t,x,y)=LE(t,y,x)$ and $\widetilde{LO}(t,x,y)=LO(t,y,x)$.
\end{lemma}
\begin{proof}

Given a stack-sortable permutation $\pi$, we consider its first letter decomposition (see Lemma~\ref{dec:231})
$$
\pi=(k\cdot\pi')\oplus\pi'',
$$ 
where $\pi'$ and $\pi''$ are smaller stack-sortable permutations (possibly empty) and $k-1$ is the length of $\pi'$. 
With this decomposition, we have
\[
\lpk_e(\pi) = 
\begin{cases}
   \,\, 1+\pk_e(\pi')+\lpk_e(\pi'') & \text{ if $k$ is even}, \\
   \,\, \pk_e(\pi')+\lpk_o(\pi'') & \text{ if $k$ is odd;} 
\end{cases}
\]
\[
\lpk_o(\pi) = 
\begin{cases}
    \,\,\pk_o(\pi')+\lpk_o(\pi'') & \text{ if $k$ is even}, \\
    \,\,\lpk_e(\pi'') & \text{ if $k=1$},\\
    \,\,1+\pk_o(\pi')+\lpk_e(\pi'') & \text{ otherwise}; 
\end{cases}
\]
\[
\pk_e(\pi) = 
\begin{cases}
   \,\, \pk_e(\pi')+\lpk_e(\pi'') & \text{ if $k$ is even}, \\
   \,\, \pk_e(\pi')+\lpk_o(\pi'') & \text{ if $k$ is odd;} 
\end{cases}
\]
and
\[
\pk_o(\pi) = 
\begin{cases}
   \,\, \pk_o(\pi')+\lpk_o(\pi'') & \text{ if $k$ is even}, \\
   \,\, \pk_o(\pi')+\lpk_e(\pi'') & \text{ if $k$ is odd.} 
\end{cases}
\]
Turning the above decomposition into generating functions yields~\eqref{eq:LE}. 
\end{proof}

\begin{proof}[{\bf Proof of Theorem~\ref{thm:gf}~(ii)}]
Utilizing the Gr\"obner basis method for the system~\eqref{eq:LE} of functional equations, we can prove that $M(t,x,y)$ satisfies the algebraic equation of degree $6$ in~\eqref{alg:gf2} via Maple. Note that system~\eqref{eq:LE} actually contains $9$ functional equations, where the four extra ones can be obtained from the last four ones in~\eqref{eq:LE} by exchanging the variables $x$ and $y$. 
\end{proof}

\newpage
\vskip0.1in
\begin{center}
\large{\bf Appendix. Two algebraic equations}
\end{center}

\vskip0.1in

The generating function $G=G(t,x,y)$ satisfies the algebraic equation
\begin{equation}\label{alg:gf1}
\a_4G^4+\a_3G^3+\a_2G^2+\a_1G+\a_0=0,
\end{equation}
where the coefficients are 
\begin{align*}
\a_4&= {t}^{3} ({t}^{2}{y}^{2}+4\,tx+4)( y{t}^{2}-{t}^{2}
+tx+1) ^{2},\\
\a_3&= 2\,{t}^{2}( {t}^{2}{y}^{2}+4\,tx+4)( y{t}^{2}-{t}^
{2}+tx+1)( 2\,{t}^{3}{y}^{2}-2\,{t}^{3}y+2\,{t}^{2}xy-t{
x}^{2}+t{y}^{2}+ty-x), \\
\a_2&=t\big( 6y^{4}(y-1)^{2}t^{8} + 12xy^{2}(y-1)(y^{2}+2y-2)t^{7} - 2t^{6}(x^{4}y^{2}-2x^{2}y^{4}-2y^{6}-3x^{4}y\\
&\quad-23x^{2}y^{3}-2y^{5}+2x^{4}+22x^{2}y^{2}-8y^{4}+24y^{3}-12y^{2}) 
-2t^{5}x(x^{4}y+x^{2}y^{3}-2y^{5}-2x^{4}\\
&\quad -12y^{4}-14x^{2}y-28y^{3}+4x^{2}+36y^{2}) 
+ t^{4}(x^{4}y^{2}-2x^{2}y^{4}+y^{6}-21x^{4}y+15x^{2}y^{3}\\
&\quad+2y^{5}+11x^{4}+11x^{2}y^{2}+21y^{4}+38x^{2}y+10y^{3}-4x^{2}-28y^{2}) + t^{3}x(5x^{4}+2y^{4}\\
&\quad+26y^{3}+10x^{2}+18y^{2}+16y-7x^{2}y^{2}-44x^{2}y) + t^{2}(15x^{4}-18x^{2}y^{2}+4y^{4}-33x^{2}y\\
&\quad+9y^{3}+3x^{2}+9y^{2}) + tx(15x^{2}-11y^{2}) + 5x^{2}-y^{2}\big),\\
\a_1 &=\big(2t^3y^2 - 2t^3y + 2t^2xy - tx^2 + ty^2 + ty - x\big)\big(2y^4(y - 1)t^6+2xy^2(y^2 + 4y - 4)t^5\\
&\quad-2t^4(x^4y - x^2y^3 - 2x^4 - 2x^2y^2 - y^4 - 4y^3 + 4y^2)-4x(x^2y - y^3 - 2x^2 - 2y^2)t^3\\
&\quad+t^2(x^4 - x^2y^2 - 2x^2y + 2y^3 + 4x^2 + 4y^2)+2tx(x - y)(x + y)+x^2 - y^2\big),\\
\a_0&= t( t^8 y^8 - 2 t^8 y^7 + 2 t^7 x y^7 + t^8 y^6 + 2 t^7 x y^6 - 2 t^6 x^4 y^4 + 3 t^6 x^2 y^6 - 8 t^7 x y^5 + 6 t^6 x^4 y^3\\
&\quad + 2 t^6 x^2 y^5 + 2 t^6 y^7 - 2 t^5 x^5 y^3 + 2 t^5 x^3 y^5 + 4 t^7 x y^4 - 4 t^6 x^4 y^2 - 4 t^6 x^2 y^4 + 2 t^6 y^6\\
&\quad + 4 t^5 x^5 y^2 - 4 t^5 x^3 y^4 + 6 t^5 x y^6 + t^4 x^8 - 2 t^4 x^6 y^2 + t^4 x^4 y^4 - 8 t^6 y^5 + 12 t^5 x^3 y^3 \\
&\quad + 4 t^5 x y^5- 5 t^4 x^4 y^3 + 5 t^4 x^2 y^5 + 4 t^6 y^4 - 8 t^5 x^3 y^2 - 8 t^5 x y^4+ 11 t^4 x^4 y^2 - t^4 x^2y^4 \\
& \quad+ 3 t^4 y^6 + 4 t^3 x^7 - 7 t^3 x^5 y^2 + 3 t^3 x^3 y^4 + 6 t^4 x^2 y^3 + 2 t^4 y^5- 4 t^3 x^3 y^3 + 4 t^3 x y^5 \\
& \quad- 4 t^4 x^2 y^2 - 4 t^4 y^4 + 10 t^3 x^3 y^2 + 2 t^3 x y^4 + 6 t^2 x^6 - 9 t^2 x^4 y^2+ 3 t^2 x^2 y^4 - t^2 x^2 y^3 \\
&\quad + t^2 y^5 + 3 t^2 x^2 y^2 + t^2 y^4 + 4 t x^5 - 5 t x^3 y^2 + t x y^4 + x^4 - x^2 y^2).
\end{align*}

\vskip0.1in
The generating function $M=M(t,x,y)$ satisfies the algebraic equation
\begin{equation}\label{alg:gf2}
\b_6M^6+\b_5M^5+\b_4M^4+\b_3M^3+\b_2M^2+\b_1M+\b_0=0,
\end{equation}
where the coefficients are
\begin{align*}
\b_6&=t^5(x - y)(t^2xy - t^2x - t^2y + t^2 - 1)^3,\\
\b_5&={t}^{4} ({x}^{3}{t}^{3}+4{x}^{2}y{t}^{3}-5x{y}^{2}{t}
^{3}-7{x}^{2}{t}^{3}+2xy{t}^{3}+5{y}^{2}{t}^{3}+6x{t}^{3}-6y
{t}^{3}-{x}^{2}{t}^{2}+
\\&\quad2xy{t}^{2}-{y}^{2}{t}^{2}-{x}^{2}t+5xyt-5
xt+yt+x-5y )( xy{t}^{2}-x{t}^{2}-y{t}^{2}+{t}^{2}-1) ^{2},\\
\b_4&={t}^{3} ( xy{t}^{2}-x{t}^{2}-y{t}^{2}+{t}^{2}-1 ) 
 ( 5{x}^{4}y{t}^{6}+5{x}^{3}{y}^{2}{t}^{6}-10{x}^{2}{y}^{3}
{t}^{6}-5{x}^{4}{t}^{6}-30{x}^{3}y{t}^{6}\\
&\quad+15{x}^{2}{y}^{2}{t}^{6}+20x{y}^{3}{t}^{6}+25{x}^{3}{t}^{6}+30{x}^{2}y{t}^{6}-45x{y}^
{2}{t}^{6}-10{y}^{3}{t}^{6}-4{x}^{3}y{t}^{5}+8{x}^{2}{y}^{2}{t}^
{5}\\
&\quad-4x{y}^{3}{t}^{5}-35{x}^{2}{t}^{6}+10xy{t}^{6}+25{y}^{2}{t}
^{6}+4{x}^{3}{t}^{5}-4{x}^{2}y{t}^{5}-4x{y}^{2}{t}^{5}+4{y}^{3
}{t}^{5}+22{x}^{2}{y}^{2}{t}^{4}\\
&\quad-2x{y}^{3}{t}^{4}+15x{t}^{6}-15y{t}^{6}-4{x}^{2}{t}^{5}+8xy{t}^{5}-4{y}^{2}{t}^{5}-5{x}^{3}
{t}^{4}-44{x}^{2}y{t}^{4}-13x{y}^{2}{t}^{4}\\
&\quad+2{y}^{3}{t}^{4}+37
{x}^{2}{t}^{4}+32xy{t}^{4}-9{y}^{2}{t}^{4}-4{x}^{2}y{t}^{3}-12
x{y}^{2}{t}^{3}-32x{t}^{4}+12y{t}^{4}+8{x}^{2}{t}^{3}\\
&\quad+8xy{t}^{3}+16{y}^{2}{t}^{3}-6{x}^{2}y{t}^{2}+2x{y}^{2}{t}^{2}-4x{t}^{3}-12y{t}^{3}+2{x}^{2}{t}^{2}-10xy{t}^{2}-4{y}^{2}{t}^{2}\\
&\quad+11x{t}^{2}+5y{t}^{2}+16xyt-4xt+4yt-2x-10y),\\
\b_3&= 2{t}^{2}\big(5{t}^{9}(y-1)^{2}( x-1)^{3}( x-2+y)(x-y)-3{t}^{8}(y-1)^{2}( x-1)^{2}(x-y)^{2}+{t}^{7}( y-1)\\
&\quad(x-1)^{2}(6{x}^{2}y+17x{y}^{2}-3{y}^{3}-16{x}^{2}-36xy+12{y}^{2}+39x-19y) -{t}^{6}( y-1)( x-1) \\
&\quad (10{x}^{2}y+x{y}^{2}+{y}^{3}-16{x}^{2}+2xy-10{y}^{2}+9x+3y) -{t}^{5}(x-1) ({x}^{3}y+9{x}^{2}{y}^{2}-6x{y}^{3} \\
&\quad-{x}^{3}-2{x}^{2}y+28x{y}^{2}+3{y}^{3}-12{x}^{2}-53xy-3{y}^{2}+41x-5y ) +{t}^{4} ( 5{x}^{3}y+17{x}^{2}{y}^{2}\\
&\quad -2x{y}^{3}-5{x}^{3}-18{x}^{2}y-16x{y}^{2}+3{y}^{3}-2{x}^{2}+21xy-7{y}^{2}+3x+y )+{t}^{3} ( {x}^{3}y\\
&\quad-3{x}^{2}{y}^{2}-6x{y}^{2}+13{x}^{2}+16xy+9{y}^{2}-23x-7y) -{t}^{2}( 5{x}^{2}y-3x{y}^{2}+13xy+3{y}^{2}\\
&\quad +x-7y)+t({x}^2 +9xy + 3x + 3y)-x-5y\big),\\
\b_2&= t\big(5{t}^{10} ( y-1) ^{2}( x-1) ^{3}( 2x+y-3)( x-y) -4{t}^{9}( y-1) ^{2}( x-1) ^{2
} ( x-y) ^{2}+{t}^{8}\\
&\quad( y-1)( x-1) ^{2}( 26{x}^{2}y+20x{y}^{2}-6{y}^{3}-46{x}^{2}-
57xy+23{y}^{2}+77x-37y) -4{t}^{7} \\ 
&\quad ( y-1)( x-1)( 6{x}^{2}y-3x{y}^{2}+{y}^{3}-8{x}^{2}+2xy-2{y}^{2}+5x-y) -{t}^{6}( x-1) (6{x}^{3}y\\
&\quad+12{x}^{2}{y}^{2}-19x{y}^{3}+{y}^{4}-6{x}^{3}+28{x}^
{2}y+57x{y}^{2}+{y}^{3}-50{x}^{2}-124xy+14{y}^{2}-26y\\
&\quad+106x)+4{t}^{5}( 5{x}^{3}y+3{x}^{2}{y}^{2}-5{x}^{3}-4{x}^{2}y-8x{y}^{2}+{y}^{3}+7xy+{y}^{2}+3x-3y) \\
&\quad+{t}^{4}(-14{x}^{2}{y}^{2}+2x{y}^{3}+6{x}^{3}+3x{y}^{2}-5{y}
^{3}+40{x}^{2}+16xy+12{y}^{2}-66x+6y) \\
&\quad-4{t}^{3}
({x}^{2}y-3x{y}^{2}+4{x}^{2}+7xy+{y}^{2}-3x-3y) +{t}^{2}(5{x}^{2}y-x{y}^{2}+3{x}^{2}+10xy\\
&\quad-{y}^{2}+23x-7y ) +t ( -8xy-4x-4y) -x+5y\big),\\
\b_1&={t}^{11}( y-1) ^{2}( x-1) ^{3}
( 5x+y-6)( x-y)-{t}^{10}t( y-1)^{2}( x-1)^{2}( x-y)^{2}\\
& \quad+{t}^{9}( y-1)(x-1) ^{2}( 21{x}^{2}y+x{y}^{2}-2{y}^{3}-31{x}^{2}-22xy+13{y}^{2}+41x-21y)   \\
&\quad-{t}^{8}( y-1)( x-1)( 9{x}^{2}y-7x{y}^{2}+2{y}^{3}-11{x}^{2}+4xy-{y}^{2}+7x-3y) \\
 &\quad-{t}^{7}(x-1)(6{x}^{3}y-7{x}^{2}{y}^{2}-8x{y}^{3}+{y}^{4}-6{x}^{3}+48{x}^{2}y+28x{y}^{2}-6{y}^{3}-46{x}^{2}\\
&\quad-86xy+28{y}^{2}+76x-28y)+{t}^{6}(10{x}^{3}y-7{x}^{2}{y}^{2}+6x{y}^{3}-{y}^{4}-10{x}^{3}+8{x}^{2}y\\
&\quad-14x{y}^{2}-2{x}^{2}+2xy+8{y}^{2}+8x-8y )-{t}^{5} (4{x}^{3}y+7{x}^{2}{y}^{2}-3x{y}^{3}-10{x}^{3}+8{x}^{2}y\\
&\quad+x{y}^{2}+{y}^{3}-28{x}^{2}-22xy+6{y}^{2}+54x-18y )+{t}^{4}
(4{x}^{2}y+3x{y}^{2}+{y}^{3}-14{x}^{2} \\
&\quad-4xy-6{y}^{2}+6x+6y) +{t}^{3}({x}^{3}+2{x}^{2}y-3x{y}^{2}+5{x}^{2}+4xy-{y}^{2}+14x-6y) \\
&\quad-{t}^{2}({x}^{2}-{y}^{2}+8x) +t(xy-x+y)+x-y,\\
\b_0&=t\big( {t}^{10}( y-1)^{2}( x-1)^{4}( x-y) +{t}^{8}( y-1)( x-1) ^{2} ( 6{x}^{2}y-2x{y}^{2}-8{x
}^{2}-3xy+3{y}^{2}\\
&+9x-5y )-{t}^{6}( x-1)(2{x}^{3}y-7{x}^{2}{y}^{2}+x{y}^{3}-2{x}^{3}+22{x}^{2}y
+3x{y}^{2}-3{y}^{3}-16{x}^{2}\\
&-24xy+12{y}^{2}+22x-10y) -{t}^{4} (2{x}^{3}y-2{x}^{2}{y}^{2}-4{x}^{3}+14{x}^{2}y-x{y}^{2}-{y}^{3}-14{x}^{2}\\
&-14xy+8{y}^{2}+22x-10y
 ) +{t}^{2}( {x}^{3}-{x}^{2}y+2{x}^{2}-4xy+2{y}^{2}+9x-5y)-x+y\big). 
\end{align*}

\section*{Acknowledgement} 
We wish to thank Yang Li for useful discussions and Qiaolong Huang for his help on computing Gr\"obner basis in Maple. This work was supported
by the National Science Foundation of China grants  12322115, 12271301 and 12071440.




\begin{thebibliography}{99}

\bibitem{bbs}M. Barnabei, F. Bonetti and M. Silimbani, The descent statistic on $123$-avoiding permutations, S\'em. Lothar. Combin.,  {\bf63} (2010), Art. B63a.

\bibitem{Cal} D. Callan,  Some bijections and identities for the Catalan and fine numbers, 
S\'em. Lothar. Combin.,  {\bf53} (2006), Article B53e.


\bibitem{Cor} S. Corteel, Crossings and alignments of permutations, Adv. in Appl. Math., {\bf38} (2007),  149--163.

\bibitem{DV} M.-P. Delest and G. Viennot, Algebraic languages and polyominoes enumeration, Theoret. Comput. Sci., {\bf34} (1984),  169--206.

\bibitem{Dum} D. Dumont, A combinatorial interpretation for the schett recurrence on the Jacobian elliptic functions, Math. Comp., {\bf33} (1979),  1293--1297.

\bibitem{EP} S. Elizalde and I. Pak, Bijections for refined restricted permutations, J. Combin. Theory Ser. A, {\bf105} (2004), 207--219.

\bibitem{FV}   J. Fran\c{c}on and G. Viennot,  Permutations selon leurs pics, creux, doubles mont\'ees et double descentes, nombres d'Euler et nombres de Genocchi (in French),  Discrete Math., {\bf28} (1979), 21--35.

\bibitem{FZ} D. Foata and D. Zeilberger,  Denert's permutation statistic is indeed Euler--Mahonian, Stud. Appl. Math., {\bf83} (1990), 31--59.

\bibitem{Knu} D.E. Knuth, The Art of Computer Programming, vol. 1, Addison-Wesley Publishing CO., Reading, Mass, 1968.

\bibitem{Kit}S. Kitaev,  Partially ordered generalized patterns,   Discrete Math., {\bf298} (2005), 212--229. 

\bibitem{Kit2} S. Kitaev, Patterns in Permutations and Words, Springer Science \& Business Media, 2011.


\bibitem{KZ} S. Kitaev and P.B. Zhang, Non-overlapping descents and ascents in stack-sortable permutations, Discrete Appl. Math., {\bf344} (2024), 112--119.

\bibitem{Kra} C. Krattenthaler, Permutations with restricted patterns and Dyck paths,  Adv. in Appl. Math., {\bf27} (2001), 510--530.



\bibitem{Lin} Z. Lin and D. Kim, Refined restricted inversion sequences,  Ann. Comb., {\bf25} (2021), 849--875.

\bibitem{LF} Z. Lin and S. Fu, On $1212$-avoiding restricted growth functions, Electron. J. Combin., {\bf24} (2017), \#P1.53.

\bibitem{LLWZ} Z. Lin, J. Liu, S. Wang and W.J.T. Zang, More bijective combinatorics of weakly increasing trees, Adv. in Appl. Math., {\bf160} (2024), Article 102755.

\bibitem{Lin3} Z. Lin and J. Liu, Proof of Dilks' bijectivity conjecture on Baxter permutations, J. Combin. Theory Ser. A, {\bf200} (2023), Article 105796.

\bibitem{LM} Z. Lin and J. Ma, A symmetry on weakly increasing trees and multiset Schett polynomials, \href{https://arxiv.org/abs/2104.10539}{arXiv:2104.10539}.

\bibitem{Lin2} Z. Lin, D.G.L. Wang and T. Zhao, A decomposition of ballot permutations, pattern avoidance and Gessel walks, J. Combin. Theory Ser. A, {\bf191} (2022), Article 105644.





\bibitem{Vie} L.-F. Pr\'eville-Ratelle and X. Viennot, The enumeration of generalized Tamari intervals, Trans. Amer. Math. Soc., {\bf369} (2017),  5219--5239.

\bibitem{RSZ} A. Robertson, D. Saracino and D. Zeilberger, Refined restricted permutations, Ann. Comb., {\bf6} (2002), 427--444.

\bibitem{Schet} A. Schett, Properties of the Taylor series expansion coefficients of the Jacobian elliptic functions, 
Math. Comp., {\bf30} (1976),  143--147.

\bibitem{SS} R.~Simion and F.~Schmidt, Restricted permutations, European J. Combin., {\bf6} (1985), 383--406.

\bibitem{EC1} R.P. Stanley, {\it Enumerative Combinatorics, Vol.~1}, second ed.,  Cambridge University Press, Cambridge, 2012.

\bibitem{Sun} Y. Sun,  A simple bijection between binary trees and colored ternary trees,  Electron. J. Combin., {\bf17} (2010), \#N20.

\end{thebibliography}
\end{document}